\documentclass[preprint,3p,11pt,authoryear]{elsarticle}
\usepackage{setspace}
\usepackage{graphicx}
\usepackage{makecell}
\usepackage[urlcolor=blue]{hyperref}
\usepackage{csquotes}
\usepackage{amsfonts}
\usepackage{subcaption}
\usepackage{amssymb}
\usepackage{amsmath, amsthm, amsfonts}
\usepackage{tabularx}
\usepackage{booktabs}

\usepackage{enumitem}
\usepackage{xcolor}

\newtheorem{proposition}{Proposition}
\theoremstyle{remark}
\newtheorem{remark}{Remark}

\begin{document}
\sloppy
\begin{frontmatter}

\title{Optimization-based lane-changing trajectory planning in mixed traffic with embedded input convex neural networks}

\author[1,2]{Yuanxiang Yang}
\ead{yuanxiang.yang@aalto.fi}

\author[3]{Yu Liu\corref{cor1}}
\ead{yu.2.liu@aalto.fi}

\author[4,3]{Fabricio Oliveira}
\ead{fabol@dtu.dk}

\author[5]{Yu Han}
\ead{yuhan2@hkust-gz.edu.cn}

\author[1,2]{Claudio Roncoli}
\ead{claudio.roncoli@kuleuven.be}

\affiliation[1]{organization={Mobility and Industrial Management (MIM), Faculty of Engineering Science, KU Leuven},
            city={Leuven},
            postcode={3001}, 
            country={Belgium}}

\affiliation[2]{organization={Department of Built Environment, Aalto University},
            city={Espoo},
            postcode={02150}, 
            country={Finland}}

\affiliation[3]{organization={Department of Mathematics and Systems Analysis, Aalto University},
            city={Espoo},
            postcode={02150}, 
            country={Finland}}
                        
\affiliation[4]{organization={DTU Management, Technical University of Denmark}, 
            city={Lyngby}, 
            postcode={2800},
            country={Denmark}}

\affiliation[5]{organization={The Hong Kong University of Science and Technology (Guangzhou)}, 
            city={Guangzhou}, 
            postcode={510642},
            country={China}}
            
\cortext[cor1]{Corresponding Author}

\begin{abstract}
Lane-changing (LC) maneuvers are highly challenging vehicle interactions, with substantial implications for traffic safety and efficiency. Developing effective LC strategies for connected and automated vehicles in mixed traffic is complicated by the uncertainty associated with surrounding-vehicle behavior and the computational burden of embedding realistic prediction models within optimization schemes. To capture these complex interactions while retaining tractability, this study proposes a data-driven optimization framework that embeds neural network~(NN) surrogates for safety and lateral comfort within a Bézier-curve-based trajectory planner. 
Specifically, the safety cost during the LC process is quantified via a surrogate composed of an input convex recurrent neural network (ICRNN) encoding the planned trajectory, a fully input convex neural network (FICNN) embedding the LC duration, and a partially input convex neural network (PICNN) decoder fusing them with historical traffic states; separately, another FICNN surrogate approximates the optimal lateral comfort cost as a function of the LC duration. 
Both surrogates are embedded via linear programming (LP) formulations, avoiding the binary variables required by conventional mixed-integer programming (MIP) approaches for NN-embedded optimization. We further theoretically establish their exactness at optimality and extend the framework to problems with additional analytical objective terms and multiple surrogates. 
Experiments on a hybrid dataset combining real-world and simulated trajectories demonstrate that the proposed method achieves surrogate accuracy competitive with unconstrained baselines while substantially improving optimization efficiency over conventional MIPs. 
Closed-loop simulations further show that the proposed planners markedly reduce the weighted planning cost and consistently outperform baseline methods. 
\end{abstract}

\begin{keyword}
Data-driven optimization \sep Lane-changing trajectory planning \sep Mixed traffic \sep Input convex neural networks \sep Surrogate modeling
\end{keyword}

\end{frontmatter}

\section{Introduction} \label{sec:intro}

The transition toward widespread vehicle automation holds significant potential to improve road safety, alleviate traffic congestion, and enhance overall transportation efficiency~\citep{papamichail2019motorway, karimi_cooperative_2020}. Realizing these benefits requires advanced control strategies capable of safely coordinating automated vehicles in complex environments. Among the different driving tasks, lane-changing (LC) trajectory planning stands as a cornerstone challenge, particularly in mixed traffic flow environments where connected automated vehicles (CAVs) co-exist with human-driven vehicles~(HDVs). In these heterogeneous traffic systems, the inherent unpredictability of human behavior introduces significant stochasticity, demanding LC trajectory planning algorithms that are not only computationally efficient but also robust and safety-critical. Considerable research has been dedicated to addressing this challenge, leading to the development of diverse techniques, which can be categorized into graph-search-based, learning-based, and optimization-based approaches~\citep{guo_survey_2023}.
 
Graph search algorithms are highly practical in complex and unstructured environments, due to their interpretability, optimality guarantee, and robust collision avoidance capabilities. However, the stability and quality of the planned trajectory are sensitive to grid resolution, and the resulting trajectories may be geometrically non-smooth or dynamically infeasible. Consequently, graph search is often paired with trajectory optimization for real-time, dynamically feasible planning~\citep{dolgov2010path, katrakazas2015real}. 

Learning-based methods have recently attracted increasing attention for trajectory planning in mixed traffic flow, because of their ability to handle complex nonlinear dynamics, uncertain interactions, and diverse driving behaviors. Typical learning-based approaches include imitation learning~\citep{le2022survey}, deep reinforcement learning~\citep{aradi_survey_2022}, and generative models~\citep{zheng2025diffusion}. Despite these advantages, learning-based methods lack interpretability and explicit theoretical guarantees, making it difficult to ensure that the generated trajectories always satisfy safety, feasibility, and traffic-rule constraints. Furthermore, they are highly data-intensive and struggle to generalize across varying environments, despite ongoing mitigation efforts~\citep{li_survey_2026}. Therefore, these limitations motivate the development of more structured and constraint-aware trajectory planning approaches.

Optimization-based methods formulate trajectory planning as a constrained optimization problem, in which the desired trajectory is obtained by minimizing a predefined objective function subject to a set of explicit constraints, such as vehicle dynamics, collision avoidance, and traffic rules. Compared with learning-based methods, their primary advantage lies in a clear mathematical structure, strong interpretability, and the ability to explicitly handle physical and safety constraints. However, optimization-based methods often face challenges in terms of computational complexity and in adequately capturing inter-vehicle interactions~\citep{teng_motion_2023}. To model the interactions between the ego vehicle and surrounding vehicles, many game theory-based trajectory planning methods have been proposed~\citep{wang_modeling_2022, yang_macro-micro_2024, yang_macro-micro_2025}. Game theory models are inherently capable of capturing interactive decision-making among traffic participants, which is particularly important in mixed traffic environments. Moreover, existing optimization-based approaches, including both model predictive control~\citep{micheli_nmpc_2023} and game-theoretic methods~\citep{ji_pricing_2023}, rely on simplified analytical safety indicators or rule-based interaction models derived from car-following principles. Although such simplifications ensure computational efficiency, they may lead to discrepancies between the predicted and actual safety costs under dynamic interactions. To address this gap, in this work we integrate a neural network (NN) model into an optimization-based trajectory planning framework. The proposed model captures interactions among vehicles and predicts potential safety costs, thereby enabling the planner to avoid overly simplified vehicle prediction models while preserving the feasibility of real-time computation.

The manner in which NN models are integrated into optimization depends critically on whether the decision variables of the optimization problem are included as input features to the trained model. To avoid ambiguity with control theory terminology, this work subsequently uses the term \textit{inputs} to refer exclusively to the features fed into the network. When the model inputs are independent of the decision variables, the resulting approach follows the standard ``predict-then-optimize'' paradigm~\citep{elmachtoub_smart_2022}, in which prediction and optimization are performed sequentially. This work instead considers the paradigm in which the model inputs include the decision variables, so that the prediction and the optimization become coupled. Such approaches are increasingly adopted in settings where first-principles models are either computationally expensive or insufficiently expressive, and have been applied in domains including engine design~\citep{yang_modeling_2022}, energy systems~\citep{chen_data-driven_2020, murzakhanov_neural_2022, wang_real-time_2025}, process systems engineering~\citep{yang_data-driven_2021,wang_fast_2024,liu_simulator-based_2025}, and manufacturing systems~\citep{dvorkin_emission-constrained_2023}.

Surrogate models adopt simpler data-driven representations to approximate complex or computationally demanding relationships, rendering them easier to incorporate into optimization.
Embedding such a trained NN into an optimization problem can be classified as black-box, gray-box, or white-box embedding, depending on the extent to which information about the NN model is exploited~\citep{dowson_mathoptaijl_2026}. Black-box embedding typically relies solely on the forward pass of the surrogate model and does not require its architecture, weights, or gradients. The resulting problem is commonly solved using derivative-free optimization algorithms, such as Bayesian optimization, which are particularly suitable when the underlying mapping is difficult to express analytically or when evaluations involve computationally expensive simulations or experiments. Such methods impose no specific restrictions on the types of decision variables or on the continuity and differentiability of the surrogate model. However, because the internal structure of the surrogate is not explicitly exploited, they generally do not provide the same deterministic global optimality guarantees for the surrogate-based optimization problem as structure-exploiting mathematical programming formulations~\citep{wang2023recent}. Gray-box embedding exploits accurate first- and second-order derivative information obtained through backpropagation and is therefore mainly applicable to continuous decision variables~\citep{elorza_casas_comparison_2025}. This approach can handle large-scale and complex NN models~\citep{lawrynczuk_input_2022, wang_real-time_2025}. However, the resulting optimization problem may converge to a local optimum and may be affected by vanishing or exploding gradients. 

Distinct from both approaches, white-box embedding fully exploits the NN model architecture by explicitly representing the forward propagation process through additional constraints incorporated into the original optimization problem. This approach supports both continuous and discrete decision variables and can provide global optimality guarantees under suitable formulations. It also offers greater flexibility in the solution process because the embedded constraints are inherently supported by existing solvers. The dominant white-box embedding approach relies on NNs with Rectified Linear Unit (ReLU) activations~\citep{huchette_when_2026}, whose piecewise-linear structure admits an exact 0--1 mixed-integer programming reformulation~\citep{fischetti_deep_2018}. Notably, the resulting MIP is NP-hard to solve, and becomes computationally intensive as the network grows. Although several acceleration techniques, including bound tightening, model compression, warm starts, and model regularization~\citep{yang_modeling_2022, wang_optimizing_2023, kenefake_multiparametric_2024, fajemisin_optimization_2024}, have been proposed, the introduction of integer variables cannot be completely avoided. This computational limitation indicates that the MIP formulation may not satisfy the real-time requirements of the trajectory planning problem in this work.

To overcome the computational bottleneck introduced by integer variables and guarantee global optimality, the input convex neural network (ICNN)~\citep{amos_input_2017} and its capability to obtain an LP formulation offer an appealing alternative when the underlying phenomenon is approximately convex. In fact, an ICNN ensures convexity with respect to its inputs by imposing non-negativity constraints on specific weights and using convex, non-decreasing activation functions. The representation capability of ICNNs is theoretically guaranteed, in that any Lipschitz continuous convex function on a compact domain can be approximated by an ICNN~\citep{chen_optimal_2019}. 
The efficiency and scalability of this specific LP reformulation substantially broaden the applicability of ICNN-based models~\citep{dvorkin_emission-constrained_2023, wu_transient_2024, liu_icnn-enhanced_2026}. Moreover, the partially input convex neural network (PICNN) introduces auxiliary variables without imposing convexity constraints, allowing the model to capture more complex input interactions while preserving convexity with respect to the decision variables. The enhanced flexibility can lead to more accurate predictions and better downstream performance compared with the more generic ICNN models, as demonstrated by~\citet{lawrynczuk_input_2022}. Given their competitive representational capacity and excellent solution efficiency when embedded into optimization problems, ICNN-based models are promising candidates for real-time trajectory planning.

Nevertheless, extending this promise to demanding tasks such as LC planning presents challenges that remain unresolved. Existing studies typically employ a single ICNN class that suffices for the system mappings they approximate, leaving the capability of ICNN-based models in more complex settings unexamined. The safety cost in LC planning illustrates such complexity, as it depends jointly on static features and dynamic trajectory sequences, which no single ICNN architecture accommodates naturally. Moreover, while recurrent extensions of ICNNs have been developed for sequential data~\citep{chen_optimal_2019, wang_real-time_2025}, their LP reformulation, along with that of surrogates composing multiple ICNN-class modules, has not been established. Bridging these gaps calls for a surrogate design that composes ICNN-class architectures to match the structure of the planning problem, together with a formal validation of the resulting LP reformulation. A preliminary version of this work~\citep{yang_itsc_2026} is extended here by theoretically establishing proofs about the validity of the proposed reformulation, presenting and comparing alternative surrogate architectures, and further strengthening the empirical evaluation using real-world trajectory data. 
Specifically, we make the following contributions:
\begin{enumerate}
    \item A data-driven, optimization-compatible framework is developed for LC trajectory planning in mixed traffic. The framework embeds NN surrogate models into an optimization-based trajectory planner to approximate computationally demanding terms. By leveraging ICNN-class models, the underlying computationally intractable trajectory planning problem is reformulated as a tractable convex optimization problem, without introducing additional binary variables.
    \item The exactness of the LP reformulation for ICNN-class models under minimization is formally established through both theoretical analysis and numerical validation. The results are further extended to objective functions containing additional analytical terms and to optimization problems involving multiple ICNN surrogates.
    \item Two architectural adaptations of ICNN-class models are proposed. A many-to-one input convex recurrent neural network (ICRNN), building on~\citet{chen_optimal_2019}, is tailored to sequence-to-scalar prediction with a streamlined structure. A PICNN variant replaces the multiplicative gating of the original architecture~\citep{amos_input_2017} with additive coupling and is empirically compared with the gated original.
    \item A structured safety surrogate is designed for optimization-based trajectory planning. The proposed model combines a PICNN with ICRNN and fully input convex neural network (FICNN) modules to integrate historical traffic information with planned trajectory features, and admits a valid LP reformulation, while maintaining competitive predictive accuracy and optimization tractability.
\end{enumerate}

The remainder of this paper is organized as follows. Section~\ref{sec:problem} formulates the LC trajectory planning problem and identifies the two objective components intractable for real-time implementation. Section~\ref{sec:framework} develops the ICNN-class surrogate-embedded optimization framework by establishing the exactness conditions for its epigraph reformulation and presenting the structured safety and lateral-comfort surrogates, as well as the resulting reformulated planning problem. Section~\ref{sec:experiments} evaluates the proposed approach in terms of surrogate prediction accuracy, reformulation efficiency, and closed-loop planning performance. Section~\ref{sec:conclusion} summarizes the main findings and discusses directions for further research.

\section{Lane-changing trajectory planning problem} \label{sec:problem}

In the mixed traffic environment, sketched in Figure~\ref{fig:game}, the ego vehicle represents the CAV executing an LC maneuver, while the corresponding obstacle vehicle refers to the following vehicle in the target lane. The LC trajectory planning is complicated, particularly when the obstacle vehicle is an HDV, since explicit cooperation cannot be assumed and the interaction is thus to be treated as stochastic. To model this, we leverage empirical trajectory data in conjunction with simulation data to model vehicle interactions, rather than relying solely on rule-based or car-following models. 

\begin{figure}[tb]
    \centering
    \includegraphics[width=0.65\textwidth]{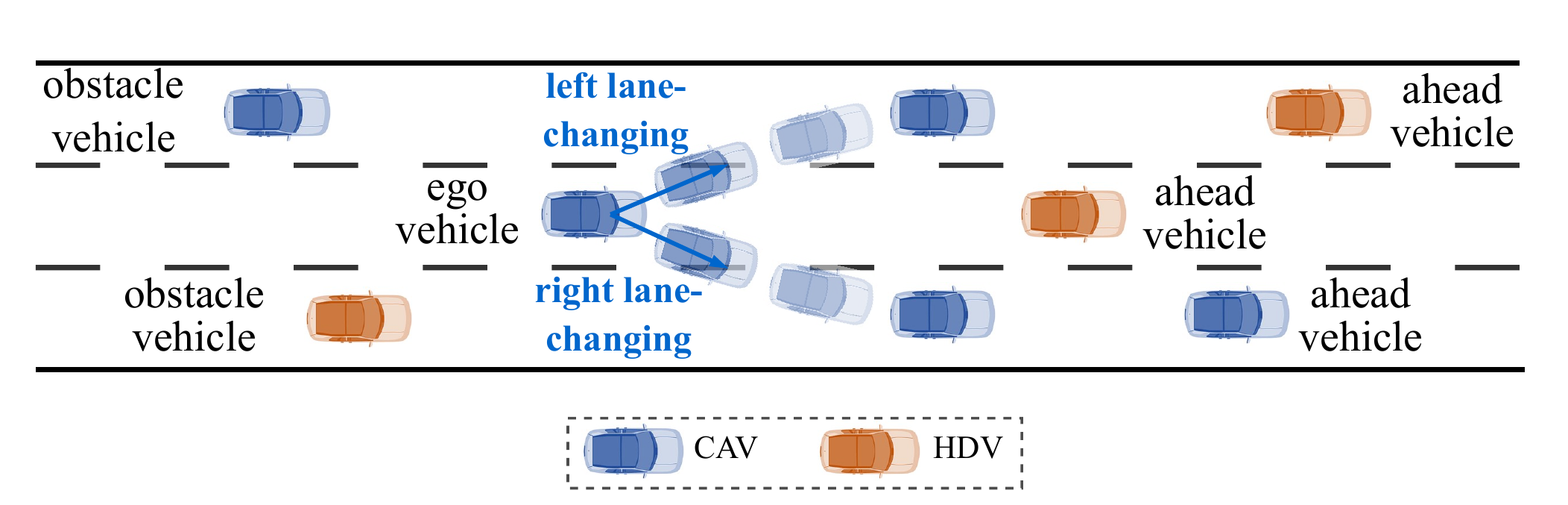}
    \caption{Schematic diagram of LC trajectory planning in mixed traffic environment, reproduced from~\citet{yang_macro-micro_2025}.}
    \label{fig:game}
\end{figure}

\subsection{Trajectory parameterization and affine representation}

In this study, the ego vehicle's trajectory during the LC process is decomposed into longitudinal and lateral trajectories, both of which are modeled through parameterized curves. Specifically, Bézier curves are employed, the applicability and advantages of which are discussed by~\citet{yang_macro-micro_2025}. Besides enforcing boundary and physical constraints through a finite number of control points, this parameterization exposes affine mappings from the control points to the discrete position, speed, and acceleration profiles. 
The decision variables for the trajectory optimization problem comprise the longitudinal and lateral trajectory control points~$(P^{\mathrm{x}}, P^{\mathrm{y}})$, as well as the LC maneuver duration~$\delta$.

\subsubsection{Longitudinal trajectory}
The longitudinal speed profile of the ego vehicle is designed based on Bézier curves of degree~$d$ as follows:
\begin{subequations}
\begin{align}
    v^{\mathrm{x}}(t) & =\sum_{i=0}^{d}{b_{i,d}\left(\frac{t}{T}\right)} P^{\mathrm{x}}_i, \qquad   0 \leq t \leq T \\
    b_{i,d}(x) & = \binom{d}{i} x^i(1-x)^{d-i},
\end{align}
\end{subequations}
where~$v^{\mathrm{x}}(t)$ denotes the longitudinal speed at time~$t$, $T$ is the length of the planning horizon, and $P^{\mathrm{x}}$ is the vector of control points. Our approach leverages the convex hull property of Bézier curves to transform the continuous consistency, safety, and physical constraints of the entire LC process into a finite set of constraints imposed on the control points:
\begin{subequations}
   \begin{align}
    & P^{\mathrm{x}}_0 = v_0^{\mathrm{x}}, \quad P^{\mathrm{x}}_d = v^{\mathrm{x}}(T),  \label{eq:vx1} \\
    & \frac{d}{T} (P^{\mathrm{x}}_1-P^{\mathrm{x}}_0) =  a_0^{\mathrm{x}}, \quad \frac{d}{T} (P^{\mathrm{x}}_d - P^{\mathrm{x}}_{d-1}) = 0, \label{eq:vx2} \\
    & 0 \leq P^{\mathrm{x}}_i \leq v_{\mathrm{max}}^{\mathrm{x}}, \ && i = 1,\ldots,d, \label{eq:vx3} \\
    & a_{\text{min}}^{\mathrm{x}} \leq \frac{d}{T}  (P^{\mathrm{x}}_{i+1}-P^{\mathrm{x}}_i) \leq a_{\text{max}}^{\mathrm{x}}, \ && i=1,\ldots,d-2. \label{eq:vx4}
    \end{align} 
\end{subequations}
Eq.~\eqref{eq:vx1} guarantees the speed consistency at the beginning and reaches the desired speed at the end of the LC process, where~$v^{\textrm{x}}_0$ and~$v^{\textrm{x}}(T)$ denote the initial and target longitudinal speeds, respectively; Eq.~\eqref{eq:vx2} ensures that the longitudinal acceleration at the start is~$a^{\textrm{x}}_0$ and reaches~0 at the end, in which~$a^{\textrm{x}}_0$ denotes the initial longitudinal acceleration; Eqs.~\eqref{eq:vx3} and~\eqref{eq:vx4} ensure that both the speed and the acceleration are maintained in the feasible range during the LC process. These constraints guarantee that the generated trajectories comply with physical and safety limitations, regardless of surrounding vehicles. The longitudinal position profile and acceleration profile can then be obtained by integrating and differentiating the speed profile, respectively. To this end, let the longitudinal trajectory be represented by the position and speed vectors~$p^{\mathrm{x}}, v^{\mathrm{x}} \in \mathbb{R}^{N}$, where~$N$ corresponds to the total number of discrete time steps of length~$\Delta t$ in the entire LC process~$T$. The degree-$d$ generic polynomial curve is given as follows:
\begin{equation}
    v^{\mathrm{x}}(t) = \alpha_0 + \alpha_1  \frac{t}{T} + \alpha_2  \Bigl(\frac{t}{T}\Bigr)^2+ \dots + \alpha_d  \Bigl(\frac{t}{T}\Bigr)^d, \quad 0 \leq t \leq T.
\end{equation}

The Bézier control points~$P$ can be mapped to the coefficient vector~$\alpha \in \mathbb{R}^{d+1}$ of a general polynomial curve via a basis transformation matrix~$M \in \mathbb{R}^{(d+1) \times (d+1)}$. This transformation represents a change of basis from the Bernstein polynomials to the monomial basis, where the entries of the transformation matrix~$M$ are uniquely determined by the degree~$d$~\citep{farin2014curves}. The transformation can be expressed as:
\begin{subequations}
    \begin{align}
        & (\alpha_0, \alpha_1, \ldots, \alpha_d)^{\top} = M(P^{\mathrm{x}}_0, P^{\mathrm{x}}_1, \ldots, P^{\mathrm{x}}_d)^\top, \\
        & M_{i,j} =
        \begin{cases}
        \binom{d}{i} \binom{i}{j} (-1)^{i - j}, & \text{if } i \ge j
        \\0, & \text{otherwise,}
        \end{cases} \quad \text{for } 0 \le i,j \le d. 
    \end{align}
\end{subequations}

Consequently, with a specific Vandermonde matrix~$V \in \mathbb{R}^{N \times (d+1)}$, the speed vector can be sampled with the time step~$\Delta t$ from polynomial curves as~\citep{golub2013matrix}:
\begin{subequations}
\begin{align}
    v^{\mathrm{x}} & = VM(P^{\mathrm{x}}_0, P^{\mathrm{x}}_1,\ldots,P^{\mathrm{x}}_d)^\top, \label{eq:vx} \\
    V & = V(\frac{\Delta t}{T},\frac{2 \Delta t}{T},\ldots,1)= \left[ \begin{matrix}
    	1&		\frac{\Delta t}{T}&		\left( \frac{\Delta t}{T} \right) ^2&		\cdots&		\left( \frac{\Delta t}{T} \right) ^d\\
    	1&		\frac{2\Delta t}{T}&		\left( \frac{2\Delta t}{T} \right) ^2&		\cdots&		\left( \frac{2\Delta t}{T} \right) ^d\\
    	\vdots&		\vdots&		\vdots&		\ddots&		\vdots\\
    	1&		1&		1^2&		\cdots&		1^d\\
    \end{matrix} \right]. \label{eq:vander}
\end{align}
\end{subequations}
Hence, the position vector~$p^{\mathrm{x}}$ can be obtained with an integration matrix~$H \in \mathbb{R}^{N \times N}$:
\begin{subequations}
\begin{align}
    p^{\mathrm{x}} & =  \mathbf{1} \, p^{\mathrm{x}}_0 + H V M (P^{\mathrm{x}}_0, P^{\mathrm{x}}_1,\ldots,P^{\mathrm{x}}_d)^\top \label{eq:px},  \\
    H_{i,j} & = \begin{cases} 
                \Delta t, & \text{if } i \ge j \\
                0, & \text{otherwise,}
                \end{cases}
\end{align}
\end{subequations}
in which,~$p^{\mathrm{x}}_0$ is the initial longitudinal position of the ego vehicle at the beginning of its LC maneuver, and $\mathbf{1}\in\mathbb{R}^N$ is the all-ones vector. With the first-order difference matrix~$D \in \mathbb{R}^{(N-1) \times N}$, the longitudinal acceleration profile can be formulated as:
\begin{subequations}
\begin{align}
    a^{\mathrm{x}} &= D V M (P^{\mathrm{x}}_0, P^{\mathrm{x}}_1,\ldots,P^{\mathrm{x}}_d)^\top, \label{eq:ax} \\
    D_{i,j} &= \frac{1}{\Delta t} \times
    \begin{cases}
        1, \quad & \mathrm{if} \quad j = i+1 \\
        -1, \quad & \mathrm{if} \quad j = i \\
        0, \quad & \mathrm{otherwise.}
    \end{cases}
\end{align}
\end{subequations}

As demonstrated in Eqs.~\eqref{eq:vx},~\eqref{eq:px}, and~\eqref{eq:ax}, the longitudinal profile of the ego vehicle can be effectively represented by control points $P$ through matrix transformations, revealing the affine relationship between the discrete trajectory and control points.

\subsubsection{Lateral trajectory}
By leveraging the endpoint interpolation property of Bézier curves, wherein the curve begins at the first control point and ends at the last, Bézier curves are utilized to model lateral position profiles as well. This approach ensures that the ego vehicle can complete LC maneuvers with precise boundary alignment. Consequently, the lateral position profile of the ego vehicle is jointly determined by the LC maneuver duration~$\delta$ and the Bézier curve control point set~$P^{\mathrm{y}}$. Consistent with the longitudinal speed profile, the lateral position profile is defined as:
\begin{equation}    
    p^{\mathrm{y}}\left( t \right) =
    \begin{cases}	
        p^{\textrm{y}}_0 + W \displaystyle\sum_{i=0}^d{b_{i,d}\left( \frac{t}{\delta} \right) P^{\textrm{y}}_i}, & 0 \le t \le \delta\\	
        p^{\textrm{y}}_0 + W, & \delta < t\leq T,\\
    \end{cases}
\end{equation}
subject to:
\begin{subequations}
\begin{align}
    & P^{\mathrm{y}}_0=0, \ P^{\mathrm{y}}_d=1, \label{eq:py1} \\
    & \frac{dW}{\delta} (P^{\mathrm{y}}_1-P^{\mathrm{y}}_0) = 0, \ \frac{dW}{\delta} (P^{\mathrm{y}}_{d}-P^{\mathrm{y}}_{d-1}) = 0, \label{eq:py2} \\
    & 0 \leq \frac{dW}{\delta} (P^{\mathrm{y}}_{i+1}-P^{\mathrm{y}}_i) \leq v^y_{\text{max}}, & \quad i = 1,\ldots,d-2, \label{eq:py3} \\
    & d (d-1) \frac{W}{\delta^2} (P^{\mathrm{y}}_2 - 2P^{\mathrm{y}}_1 + P^{\mathrm{y}}_0)=0, \label{eq:py4} \\
    & d (d-1) \frac{W}{\delta^2} (P^{\mathrm{y}}_d - 2P^{\mathrm{y}}_{d-1} + P^{\mathrm{y}}_{d-2})=0, \label{eq:py5} \\
    & a^{\mathrm{y}}_{\text{min}} \leq d (d-1) \frac{W}{\delta^2} (P^{\mathrm{y}}_{i+2} - 2P^{\mathrm{y}}_{i+1} + P^{\mathrm{y}}_i) \leq a^{\mathrm{y}}_{\text{max}}, & \quad i=1,\ldots,d-3, \label{eq:py6} \\
    & 0 \leq P^{\mathrm{y}}_i \leq 1, & \quad i = 1,\ldots,d-1, \label{eq:py7} \\
    & \delta_{\mathrm{min}} \le \delta \le T, \label{eq:py8}
\end{align}
\end{subequations}
where~$p^{\textrm{y}}_0$ and~$W$ denote the initial lateral position and the lane width, respectively. Eq.~\eqref{eq:py1} guarantees the consistency of the lateral trajectory and forces the vehicle to reach the middle of the target lane at the end of the LC process; Eq.~\eqref{eq:py2} ensures that the initial and terminal lateral speeds are both 0, avoiding consecutive LC maneuvers; Eqs.~\eqref{eq:py3} and~\eqref{eq:py6} ensure that the lateral speed and acceleration during the LC process are maintained within feasible ranges; Eqs.~\eqref{eq:py4} and~\eqref{eq:py5} impose that the initial and terminal lateral acceleration is 0; Eq.~\eqref{eq:py8} is set to guarantee that the LC maneuvers can be executed within~$T$ and larger than the minimum safe value~$\delta_{\mathrm{min}}$. Additionally, the control points of the lateral and longitudinal trajectories of the same ego vehicle are not necessarily identical. More details about constraints can be found in the work by~\citet{yang_macro-micro_2025}. Following the same sampling procedure as in Eq.~\eqref{eq:vx}, the lateral position profile can be expressed as: 
\begin{subequations}
\begin{align}
    p^{\mathrm{y}}\left( t \right) & =
    \begin{cases}	
        p^{\textrm{y}}_0 + W V(\delta) M (P^{\mathrm{y}}_0, P^{\mathrm{y}}_1,\ldots,P^{\mathrm{y}}_d)^\top, & 0 \le t \le \delta\\	
        p^{\textrm{y}}_0 + W, & \delta < t\leq T,
    \end{cases} \\
    V(\delta) & = V(\frac{\Delta t}{\delta},\frac{2 \Delta t}{\delta},\ldots,1). \label{eq:vandelta}
\end{align}
\end{subequations}
The size of matrix~$V(\delta) \in \mathbb{R}^{\frac{\delta}{\Delta t} \times (d+1)}$ is determined by the degree of Bézier curves and the LC maneuver duration. Let~$D^{(2)} \in \mathbb{R}^{(\frac{\delta}{\Delta t} -2) \times \frac{\delta}{\Delta t}    }$ denote the second-order difference matrix. The lateral acceleration profile of the ego vehicle can then be derived by differentiating twice the lateral position profile. Finally, the lateral speed and acceleration profile can be expressed as:
\begin{subequations}
\begin{align}
    a^{\mathrm{y}}\left( t \right) & =
    \begin{cases}	
        W D^{(2)} V(\delta) M (P^{\mathrm{y}}_0, P^{\mathrm{y}}_1,\ldots,P^{\mathrm{y}}_d)^\top, & 0 \le t \le \delta\\	
        0, & \delta < t\leq T
    \end{cases} \label{eq:ay} \\
    D^{(2)}_{i,j} & = \frac{1}{(\Delta t)^2} \times
    \begin{cases}
        1, \quad & \mathrm{if} \quad j = i \\
        -2, \quad & \mathrm{if} \quad j = i+1 \\
        1, \quad & \mathrm{if} \quad j = i+2 \\
        0, \quad & \mathrm{otherwise.}
    \end{cases} \label{eq:second_order_D}
\end{align}
\end{subequations}

\subsection{Objective function components and computational bottlenecks} \label{sec:objective_components}
In the context of optimization-based trajectory planning, the objective function typically comprises three metrics, namely driving safety, ride comfort, and travel efficiency, whose positively weighted sum is minimized~\citep{yang_macro-micro_2025}. The specific definitions for each of these cost terms are provided below.

\textrm{(i)} Safety cost~$C_{\text{safety}}$:~\citet{wang_review_2021} divided the safety measures commonly used in the CAV safety modeling into three categories: time-based, deceleration-based, and energy-based. Although the fundamentals of these indices differ, they share a common implementation challenge related to incorporating them effectively into optimization-based trajectory planning problems. In this work, the safety cost predicted for the entire LC process will be approximated by a surrogate model built from ICNN-class modules. 

The surrogate model predicts the safety cost based on static features, historical trajectories, and the planned LC trajectory of the ego vehicle. The historical trajectories include the longitudinal positions and speeds of the ego vehicle, the obstacle vehicle, and the ahead vehicles in both the initial and target lanes during a fixed period before the LC maneuver begins. As noted by~\citet{wang_review_2021}, Time Exposed Time-to-Collision (TET) and Time Integrated Time-to-Collision (TIT) are commonly utilized in the CAV longitudinal trajectory optimization objectives as safety indicators. Compared to instantaneous Time-to-Collision (TTC) alone, these safety measures provide an aggregated risk over the entire lane change process, rendering them more appropriate for microscopic simulation and this work. Other energy- and deceleration-based metrics cannot be readily computed due to the absence of mass and deceleration capacity parameters in the trajectory dataset. Hence, the TET and TIT are adopted to quantify the safety cost between vehicle~$i$ and its front vehicle~$j$, which can be formulated at each time instant~$t$ of the LC process as follows:
\begin{subequations}
    \begin{align}
        & \mathrm{TTC}(i,j,t) = \frac{p^{\textrm{x}}_j(t) - p^{\textrm{x}}_i(t) - l_j}{v^{\textrm{x}}_i (t) -v^{\textrm{x}}_j ( t)},  \\
        &\mathrm{TET}(i,j,t) = \begin{cases}
        1, \quad 0 < \mathrm{TTC}(i,j,t) < \mathrm{TTC}^\star, \\
        0, \quad \mathrm{otherwise}, \end{cases}  \label{eq:tet} \\
        & \mathrm{TIT}(i,j,t) = \frac{1}{\mathrm{TTC}(i,j,t)} - \frac{1}{\mathrm{TTC}^\star}, \quad
        0 < \mathrm{TTC}(i,j,t) < \mathrm{TTC}^\star, \label{eq:tit}
    \end{align}
\end{subequations}
where~$p^{\textrm{x}}_i$,~$p^{\textrm{x}}_j$,~$v^{\textrm{x}}_i$ and~$v^{\textrm{x}}_j$ are the longitudinal positions and speed of vehicles~$i$ and~$j$, respectively;~$l_j$ denotes the length of the front vehicle~$j$; and~$\mathrm{TTC}^\star$ denotes the TTC threshold value. For the ego vehicle~$i$, its safety cost is jointly determined by the obstacle vehicle~$j$, and ahead vehicles in both initial lane~$i^{\mathrm{ahead}}$ and target lane~$j^{\mathrm{ahead}}$ shown in Figure~\ref{fig:game}. Additionally, the safety cost is also affected by the lateral positions of vehicles. Specifically, the ahead vehicle in the initial lane~$i^{\mathrm{ahead}}$ would not pose any safety cost after the completion of the ego vehicle's LC maneuver. Therefore, following previous work~\citep{yang_macro-micro_2025}, the safety cost formulation is defined as:%
\begin{equation} \label{eq:safety}
    C_\mathrm{safety} = \sum_{n=1}^{\frac{\delta}{\Delta t}} f_{\mathrm{safety}}(i,i^{\mathrm{ahead}} ,n \Delta t)  + 
     \sum_{n=1}^{N} \frac{1}{2} \left( f_{\mathrm{safety}}(j,i,n \Delta t) + f_{\mathrm{safety}}(i,j^{\mathrm{ahead}} ,n \Delta t) \right),
\end{equation}
where~$f_{\mathrm{safety}}$ represents the safety measurement function, which is quantified in this study using TET or TIT. A weighting factor of~$\frac{1}{2}$ is applied to vehicles in the target lane, ensuring the numerical balance between the initial and target lane safety considerations. Regardless of whether TET or TIT is selected, the resulting safety function introduces significant computational complexity, rendering direct inclusion in the trajectory optimization problem intractable for real-time applications. Approximating $C_\mathrm{safety}$ with a trained FICNN surrogate replaces this analytically cumbersome term with an explicit, optimization-compatible representation, making real-time trajectory optimization feasible.

\textrm{(ii)} Comfort cost $C_{\text{com}}$: Ride comfort reflects the smoothness of the ego vehicle's motion during the LC maneuver and is evaluated based on both longitudinal and lateral accelerations:
\begin{subequations}\label{eq:comfort}
\begin{align}
    C_{\mathrm{com}}
    &= C_{\mathrm{lon}}(P^{\mathrm{x}})+C_{\mathrm{lat}}(P^{\mathrm{y}},\delta), \label{eq:comfort:decomposition}\\
    C_{\mathrm{lon}}(P^{\mathrm{x}})
    &= \sum_{n=1}^{N} a^{\mathrm{x}}(n\Delta t)^2, \label{eq:comfort:longitudinal}\\
    C_{\mathrm{lat}}(P^{\mathrm{y}},\delta)
    &= \sum_{n=1}^{\frac{\delta}{\Delta t}} a^{\mathrm{y}}(n\Delta t)^2, \label{eq:comfort:lateral}
\end{align}
\end{subequations}
in which,~$C_{\mathrm{lon}}$ and~$C_{\mathrm{lat}}$ denote the cumulative longitudinal and lateral comfort costs, respectively. Both terms are formulated as convex quadratic expressions to ensure computational efficiency. However, the inclusion of the LC maneuver duration~$\delta$ as a decision variable introduces a structural challenge to the lateral comfort cost~$C_{\mathrm{lat}}$. Specifically, the dimensionalities of the matrices in Eqs.~\eqref{eq:vandelta} and~\eqref{eq:second_order_D} are dependent on this decision variable, complicating the implementation with standard fixed-size optimization solvers.

Notably, the lateral comfort cost~$C_{\mathrm{lat}}$ depends only on the LC duration~$\delta$ and lateral trajectory control points~$P^{\mathrm{y}}$. Since~$P^{\mathrm{y}}$ appears in neither the remaining objective components nor constraints, it can be eliminated exactly by partial minimization. Accordingly, the optimal lateral comfort cost function is defined as follows:
\begin{equation}\label{eq:lateral_value_function}
    C_{\mathrm{lat}}^{\star}(\delta)
    := \min_{P^{\mathrm{y}}}\left\{
    C_{\mathrm{lat}}(P^{\mathrm{y}},\delta)
    \;\middle|\;
    P^{\mathrm{y}}\in\mathcal{P}_{\mathrm{y}}(\delta)
    \right\},
\end{equation}
where $C_{\mathrm{lat}}^{\star}(\delta)$ is the minimum achievable lateral comfort cost for an LC maneuver lasting~$\delta$; and~$\mathcal{P}_{\mathrm{y}}(\delta)$ is the set of lateral trajectories satisfying the boundary, speed, and acceleration constraints. The resulting optimal cost~$C_{\mathrm{lat}}^{\star}(\delta)$ can therefore be regarded as a function of~$\delta$ alone, which scales proportionally to~$\frac{1}{\delta^3}$ when minimizing the sum of squared lateral accelerations~\citep{werling_optimal_2010}. This property motivates the substitution of~$C_{\mathrm{lat}}$ in Eq.~\eqref{eq:comfort} with a surrogate model, similar to the one used for the safety metric. The specific architectures of both surrogate models are detailed in Section~\ref{sec:surrogate}. 

\textrm{(iii)} Efficiency cost $C_{\text{eff}}$: The change in travel distance caused by the speed adjustment during the entire interval is used to assess the degree of speed fluctuation. Hence, the efficiency measurement function is expressed as:
\begin{equation} \label{eq:efficiency}
    C_{\text{eff}} = \sum_{n = 1}^{N}{(v^{\mathrm{x}}_0 - v^{\mathrm{x}}(n \Delta t)) \Delta t},
\end{equation}
which is an affine function of the longitudinal trajectory control points~$P^{\mathrm{x}}$. This linear relationship ensures that this term is computationally straightforward to implement within the optimization framework.

\subsection{Surrogate-free trajectory optimization problem} \label{sec:original_problem}
We first formulate the trajectory optimization problem without any learned component, aiming to minimize the weighted sum of safety, comfort, and efficiency costs subject to the physical trajectory constraints:%
\begin{subequations}\label{eq:original_planning}
\begin{align}
    \min_{P^{\mathrm{x}},P^{\mathrm{y}},\delta}\quad
    & w_{\mathrm{s}} C_{\mathrm{safety}}(P^{\mathrm{x}},\delta)
      + w_{\mathrm{c}}\left[C_{\mathrm{lon}}(P^{\mathrm{x}})
      +C_{\mathrm{lat}}(P^{\mathrm{y}},\delta)\right]
      + w_{\mathrm{e}} C_{\mathrm{eff}}(P^{\mathrm{x}}) \label{eq:original_planning:obj}\\
    \mathrm{s.t.}\quad
    & P^{\mathrm{x}}\in\mathcal{P}_{\mathrm{x}}, \label{eq:original_planning:px}\\
    & P^{\mathrm{y}}\in\mathcal{P}_{\mathrm{y}}(\delta),\qquad
      \delta_{\min}\leq\delta\leq T, \label{eq:original_planning:py}
\end{align}
\end{subequations}
in which, the positive coefficients~$w_{\mathrm{s}},w_{\mathrm{c}},w_{\mathrm{e}}$ denote weights assigned to safety, comfort, and efficiency costs. The sets~$\mathcal{P}_{\mathrm{x}}$ and~$\mathcal{P}_{\mathrm{y}}(\delta)$ contain longitudinal and lateral control points satisfying defined constraints, respectively.

For any LC duration~$\delta$, the planner will choose the feasible lateral control points that give the smallest lateral discomfort, because $P^{\mathrm{y}}$ appears only in the lateral comfort term and its associated feasibility conditions. The same planning decision can therefore be expressed using the optimal lateral comfort~$C_{\mathrm{lat}}^{\star}(\delta)$:
\begin{subequations}\label{eq:reduced_original_planning}
\begin{align}
    \min_{P^{\mathrm{x}},\delta}\quad
    & w_{\mathrm{s}} C_{\mathrm{safety}}(P^{\mathrm{x}},\delta)
      + w_{\mathrm{c}}\left[C_{\mathrm{lon}}(P^{\mathrm{x}})
      +C_{\mathrm{lat}}^{\star}(\delta)\right]
      + w_{\mathrm{e}} C_{\mathrm{eff}}(P^{\mathrm{x}}) \label{eq:reduced_original_planning:obj}\\
    \mathrm{s.t.}\quad
    & P^{\mathrm{x}}\in\mathcal{P}_{\mathrm{x}},\qquad
      \delta_{\min}\leq\delta\leq T. \label{eq:reduced_original_planning:domain}
\end{align}
\end{subequations}

Although the problem~in~Eq.~\eqref{eq:reduced_original_planning} is a simplification of that in~Eq.~\eqref{eq:original_planning}, two objective components still lack representations suitable for real-time solution, since the aggregate safety cost~$C_{\mathrm{safety}}$ is determined by complex interactions between the planned ego trajectory and the observed traffic context, while $C_{\mathrm{lat}}^{\star}$ is the value of the duration-dependent lateral trajectory problem in Eq.~\eqref{eq:lateral_value_function}. Evaluating either term directly for every candidate $(P^{\mathrm{x}},\delta)$ would undermine the computational efficiency of the planner. We therefore design the surrogates $\hat C_{\mathrm{safety}}$ and $\hat C_{\mathrm{lat}}^{\star}$, which are trained offline from trajectory and optimization data, so that they can be embedded explicitly into Eq.~\eqref{eq:reduced_original_planning}.

\section{ICNN-class surrogate-embedded optimization framework} \label{sec:framework}

This section presents the ICNN-class surrogate models and their explicit embedding into the trajectory planning problem. We first introduce the MIP formulation of a conventional ReLU MLP and the LP formulation of FICNNs. The exactness of the LP representation is subsequently established, followed by the structured safety surrogate, the lateral comfort surrogate, and the final surrogate-embedded planning formulation. 

\subsection{Embedding neural network surrogates in optimization}

To introduce the neural network embedding formulations, consider a generic optimization problem with a single objective function:
\begin{subequations}\label{eq:generic_optimization_problem}
\begin{align}
    \min_{x}\quad & f(x)\\
    \mathrm{s.t.}\quad & x\in\mathcal{X},
\end{align}
\end{subequations}
where $x$ is the decision vector, $\mathcal{X}$ is the feasible region defined by the original problem constraints, and $f(x)$ is an objective function that is unavailable or unsuitable for direct optimization. When $f(x)$ is approximated by a trained neural network surrogate~$\hat f(x)$, the surrogate replaces $f(x)$ in the objective while the original feasible region~$\mathcal{X}$ is retained. Embedding the surrogate then requires its forward pass to be expressed explicitly through additional constraints. For clarity, this subsection considers the minimization problem with one surrogate output. 

\subsubsection{Conventional MIP reformulation}

For an MLP with~$L$ layers, consisting of~$L-1$ hidden layers and one output layer:
\begin{subequations}
\begin{align}
h^{(l)} & = \sigma(W^{(l)}h^{(l-1)} + b^{(l)}), \quad l = 1,\ldots, L-1 \\
h^{(L)} & = W^{(L)}h^{(L-1)} + b^{(L)},
\end{align}
\end{subequations}
where~$h^{(0)}$ and~$h^{(l)}$ denote the input vector and the activation vectors of the $l$-th hidden layer, respectively. $W^{(l)}$ and~$b^{(l)}$ are the weights and biases of the~$l$-th layer, and $\sigma^{(l)}$ represents the activation function of the~$l$-th layer. With ReLU hidden layers, $h^{(l)}=\max\{0,W^{(l)}h^{(l-1)}+b^{(l)}\}$. Among various activation functions, ReLU is most frequently employed in the relevant studies due to its piecewise-linear structure that is amenable to reformulation through linear constraints, while maintaining competitive predictive accuracy~\citep{huchette_when_2026}. Accordingly, the standard big-$M$ representation is
\begin{subequations}\label{eq:fnn_mip}
\begin{align}
    \min \quad &\hat{f}(x) \\
    \mathrm{s.t.}\quad 
    &h^{(0)}=x, \\
    &h^{(l)}-s^{(l)}=W^{(l)} h^{(l-1)} + b^{(l)}, \quad && l=1,\ldots,L-1 \label{eq:fnn_mip:con1}\\
    &0\le h^{(l)}_{j}\le U^{(l)}_{j} z^{(l)}_{j},\quad && l=1,\ldots,L-1, \ j=1,\ldots,n_l \label{eq:fnn_mip:con2}\\
    &0\le s^{(l)}_{j}\le -L^{(l)}_{j} (1-z^{(l)}_{j}),\quad && l=1,\ldots,L-1, \ j=1,\ldots,n_l \label{eq:fnn_mip:con3}\\
    &z^{(l)}_{j}\in \{0,1\},\quad && l=1,\ldots,L-1, \ j=1,\ldots,n_l \label{eq:fnn_mip:con4}\\
    &\hat{f}(x)=W^{(L)} h^{(L-1)} + b^{(L)}, \label{eq:fnn_mip:con5}\\
    &L^{(L)}\le \hat{f}(x)\le U^{(L)}, \label{eq:fnn_mip:con6} \\
    & x \in \mathcal{X},
\end{align}
\end{subequations}
in which, variables~$s^{(l)}_{j}$ and~$z^{(l)}_{j}$ denote the slack and binary variables for the~$j$-th neuron in the~$l$-th layer. Constraints~\eqref{eq:fnn_mip:con1}--\eqref{eq:fnn_mip:con6} are introduced by reformulating the trained NN model through the big-$M$ method~\citep{fischetti_deep_2018, grimstad_relu_2019}; constraints~\eqref{eq:fnn_mip:con1}--\eqref{eq:fnn_mip:con4} encode the ReLU activations in the hidden layers; constraint~\eqref{eq:fnn_mip:con5} defines the affine output layer of the trained network; while constraint~\eqref{eq:fnn_mip:con6} imposes valid bounds on the network output. 

The constants~$L$ and~$U$ denote valid lower and upper bounds for the corresponding neuron values. The tightness of these bounds has a substantial effect on the strength of the LP relaxation and, consequently, on the computational performance of the MIP. Nevertheless, as discussed in Section~\ref{sec:intro}, acceleration techniques mitigate but do not eliminate the computational burden of the binary activation variables, motivating the use of ICNN-class models that admit reformulation without introducing integer variables.

\subsubsection{FICNN and LP formulation}
The forward pass of a fully input convex neural network (FICNN), as shown in Figure~\ref{fig:icnn}, can be formulated as:
\begin{subequations}\label{eq:icnn_activations}
    \begin{align}
        h^{(1)} & = \sigma(S^{(1)} h^{(0)} + b^{(1)}), \\
        h^{(l)} & = \sigma(W^{(l)} h^{(l-1)} + S^{(l)} h^{(0)} + b^{(l)}), \quad l = 2,\ldots,L-1 \\
        h^{(L)} & = W^{(L)} h^{(L-1)} + S^{(L)} h^{(0)} + b^{(L)}.
    \end{align}
\end{subequations}

\begin{figure}[tb]
    \centering
    \includegraphics[width=0.65\linewidth]{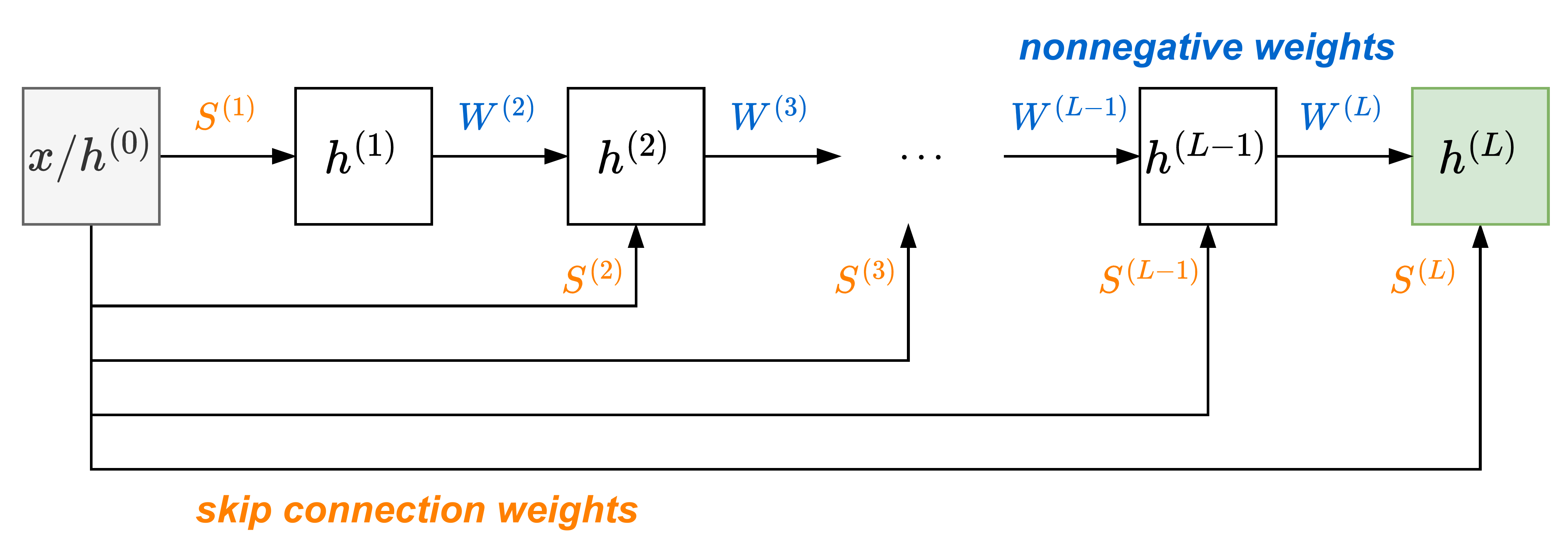}
    \caption{Architecture of a generic $L$-layer fully input convex neural network (FICNN).}
    \label{fig:icnn}
\end{figure}
The key architectural distinctions between FICNN and general MLP are as follows:

\textrm{(i)} Limitations on weights $W$ and activation functions $\sigma$: FICNNs guarantee convexity by requiring the inter-layer weights~$W$ to be non-negative and the activation functions~$\sigma$ to be convex and non-decreasing. These requirements arise from the composition rule for convex functions, where $h(x) = f(g(x))$ is convex in $x$ if $f$ is both convex and non-decreasing, and $g$ is convex. This ensures that each layer preserves the convexity of its input, thereby guaranteeing end-to-end convexity of the network output. Among the activation functions satisfying these criteria, ReLU is the most widely adopted in ICNN-class models (a survey of suitable alternatives is provided by~\citealp{lawrynczuk_input_2022}). It should be emphasized that these constraints constitute sufficient but not necessary conditions for convexity, with necessary and sufficient conditions for ReLU NNs established by~\citet{gagneux_convexity_2025}.

\textrm{(ii)} Skip connections with weights $S$: Direct skip connections from the input to each hidden layer are introduced to enhance the representational capacity of FICNN. Unlike the weights~$W$, the associated skip connection weights~$S$ are unconstrained, since the skip connections introduce only linear terms in inputs, which do not violate convexity.

By replacing the original objective function with the trained FICNN model with ReLU activations, the optimization problem can be converted into the following form via LP reformulation~\citep{amos_input_2017}:
\begin{subequations}\label{eq:icnn_lp}
\begin{align}
    \text{min} &\quad \hat{f}(x)\\
    \text{s.t.} 
    &\quad h^{(0)} =x, \\
    &\quad h^{(1)} \geq S^{(1)} h^{(0)} + b^{(1)}, \label{eq:icnn_lp:con2}\\
    &\quad h^{(l)} \geq W^{(l)} h^{(l-1)} + S^{(l)} h^{(0)} + b^{(l)}, \quad & l = 2,\ldots,L-1 \label{eq:icnn_lp:con3}\\
    &\quad h^{(l)} \geq 0, \quad & l = 1,\ldots,L-1 \label{eq:icnn_lp:con4}\\
    &\quad \hat{f}(x) = W^{(L)} h^{(L-1)} + S^{(L)} h^{(0)} + b^{(L)}, \label{eq:icnn_lp:con5}\\
    &\quad x \in \mathcal{X}. \label{eq:icnn_lp:con1}
\end{align}
\end{subequations}

As shown in Eq.~\eqref{eq:icnn_lp}, the LP reformulation for ICNN-class models relies solely on linear constraints, without introducing nonlinear constraints or integer variables. The essence of this reformulation is to replace the exact ReLU activation equations with their \textit{epigraph} representation. In this sense, constraints~\eqref{eq:icnn_lp:con2}--\eqref{eq:icnn_lp:con4} can be interpreted as the relaxation of the ReLU activation constraints. As a result, a feasible solution of Eq.~\eqref{eq:icnn_lp} may contain hidden-layer values that do not coincide with the activations produced by the trained FICNN under a forward pass. The next subsection addresses this distinct exactness question.

\subsection{Exactness of the ReLU epigraph reformulation}
Replacing ReLU equalities by epigraph inequalities enlarges the feasible region, because the hidden-layer variables may lie strictly above the corresponding values generated by the surrogate forward pass. However, this slack can be removed without worsening the objective under direct minimization of the FICNN output. Then LP reformulation is therefore tight at optimality, as formally established in Proposition~\ref{prop:icnn_lp_tightness}. 
The exactness of the epigraph reformulation is also analyzed by~\citet{liu_icnn_2026}, covering both fixed inputs and direct output minimization through its equivalence with the LP relaxation of an MIP reformulation. Here, we focus on the direct output minimization setting relevant to the trajectory planning problem considered in this work, and prove the tightness through a slack-tightening argument directly on the epigraph constraints.

\begin{proposition}[Tightness of the FICNN LP reformulation]
\label{prop:icnn_lp_tightness}
Consider the LP reformulation of a trained FICNN with ReLU activations shown in Eq.~\eqref{eq:icnn_lp}. Suppose that the hidden layer weights satisfy $W^{(l)} \geq 0$ for $l=2,\ldots,L$, and that the FICNN output is minimized with a positive coefficient in the objective. Then the LP reformulation is tight at optimality, in the sense that there exists an optimal solution in which every hidden neuron satisfies the exact ReLU activation $h^{(l)}=\max\{a^{(l)},0\}$, where $a^{(1)}=S^{(1)}h^{(0)}+b^{(1)}$ and $a^{(l)}=W^{(l)}h^{(l-1)}+S^{(l)}h^{(0)}+b^{(l)}$ for $l=2,\ldots,L-1$. Consequently, the LP reformulation gives the same optimal value as the forward pass output of the trained FICNN at the same optimal input.
\end{proposition}

\begin{proof}
For each hidden neuron, constraints~\eqref{eq:icnn_lp:con2}--\eqref{eq:icnn_lp:con4} impose $h^{(l)} \geq a^{(l)}$ and $h^{(l)} \geq 0$, and therefore $h^{(l)} \geq \max\{a^{(l)},0\}$. Thus, these constraints describe the epigraph of the composition of the ReLU activations rather than the exact ReLU graph. Suppose that, at an optimal solution, some hidden neurons are strictly above this epigraph boundary, i.e., $h^{(l)}>\max\{a^{(l)},0\}$. Then $h^{(l)}$ can be decreased to $\tilde h^{(l)}=\max\{a^{(l)},0\}$ while preserving feasibility of the corresponding constraint pair. Since all downstream hidden layer weights are non-negative and all activation functions are non-decreasing, decreasing a hidden-layer variable cannot increase any subsequent pre-activation value, and hence cannot increase the final FICNN output. Applying this tightening argument successively through the hidden layers yields a feasible solution satisfying $\tilde h^{(l)}=\max\{\tilde a^{(l)},0\}$ for all $l=1,\ldots,L-1$, with resulting output no larger than before.

Therefore, any slack in the ReLU epigraph constraints can be removed without worsening the objective. If tightening the slack strictly decreases the objective value, then optimality is contradicted. Otherwise, the tightened solution attains the same objective value and is therefore also optimal. Hence, the LP reformulation is tight at optimality and correctly represents the trained FICNN surrogate under direct output minimization.
\end{proof}

Proposition~\ref{prop:icnn_lp_tightness} considers the direct minimization of a single FICNN output. However, the trajectory planning problem in Eq.~\eqref{eq:reduced_original_planning} has a weighted objective composed of analytical terms and two terms approximated by surrogates. In particular, the longitudinal comfort and efficiency costs remain analytical, whereas the safety cost and the optimal lateral comfort value are represented by separate surrogate models. The following remark states how the tightness result extends to this objective structure.

\begin{remark}[Extension to analytical objective terms and multiple surrogates]
\label{rem:objective_extension}
Consider an objective of the form
\begin{equation}\label{eq:multiple_surrogate_objective}
    g(x)+\sum_{k=1}^{K}\lambda_k\hat f_k(x),
    \qquad \lambda_k>0,
\end{equation}
where $g(x)$ does not depend on the internal variables of the surrogate models, and each $\hat f_k$ has an epigraph representation satisfying the non-negative weight conditions of Proposition~\ref{prop:icnn_lp_tightness}. Tightening the internal variables of surrogate~$k$ leaves $x$, $g(x)$, and the internal variables of the other surrogates unchanged, while it cannot increase $\hat f_k(x)$. Since $\lambda_k>0$, this operation does not increase the overall objective. The tightening argument can therefore be applied independently to each surrogate. Consequently, there exists an optimal solution in which the internal variables and output of every embedded surrogate coincide with their corresponding forward-pass values.
\end{remark}

Remark~\ref{rem:objective_extension} shows that the entire objective function need not be approximated by a surrogate model. Only computationally intractable or analytically unavailable components need to be replaced by surrogates, while the remaining components can be retained in analytical form. It also establishes that the LP reformulation remains exact at optimality for multiple FICNN surrogates, provided that each surrogate output enters the minimization objective with a positive coefficient. Thus, the trained surrogates considered in this work can be represented exactly at optimality without introducing binary variables.

Proposition~\ref{prop:icnn_lp_tightness} and Remark~\ref{rem:objective_extension} establish the conditions under which the validity of the LP reformulation is guaranteed. The optimization problem considered in this work satisfies these conditions, so the LP reformulation exactly represents the trained surrogates at optimality without further machinery. For settings beyond these assumptions, we refer to the LP-driven branch-and-bound approach by~\citet{liu_icnn_2026} as a complementary route.

\subsection{Extensions of input convex neural networks}
The FICNN introduced above provides the basic input convex architecture and corresponding LP reformulation. However, the safety surrogate required in this work must process two types of inputs: the planned LC trajectory, which is sequential and depends on the planning decisions, and supplementary traffic features, such as historical vehicle trajectories and static vehicle information, for which convexity is not required. Hence, this section introduces two extensions of the input-convex architecture: ICRNN for sequential trajectory inputs and PICNN for incorporating additional traffic information while preserving convexity with respect to the decision variables.

\subsubsection{Input convex recurrent neural network}

To capture the inherent temporal dependencies in vehicle trajectory data, we develop a many-to-one ICRNN model, building on the architectural principles introduced by~\citet{chen_optimal_2019}. The model architecture is illustrated in Figure~\ref{fig:icrnn}, with the corresponding forward pass defined as:%
\begin{subequations}
    \begin{align} 
        & h_0 = 0, \\
        & h_{t} = \sigma(W_h h_{t-1} + W_x x_t + b_h), \quad t = 1,2,\dots,T, \\
        & y = W_y h_T + S x + b_y, 
    \end{align}
\end{subequations}
where~$T$ is the length of the time series data, and $x = (x_1^\top, \dots, x_T^\top)^\top$ denotes the concatenated input sequence. To ensure the convexity of the output in $x$ and allow for an LP reformulation, the activation function~$\sigma$ is defined as ReLU, while the weight matrices $W_h$ and $W_y$ are constrained to be non-negative. Compared with generic many-to-one recurrent neural network (RNN) architectures, the proposed ICRNN incorporates non-negative weight constraints to enforce convexity, alongside skip connections designed to improve representational capacity. 

\begin{figure}[tb]
    \centering
    \includegraphics[width=0.70\linewidth]{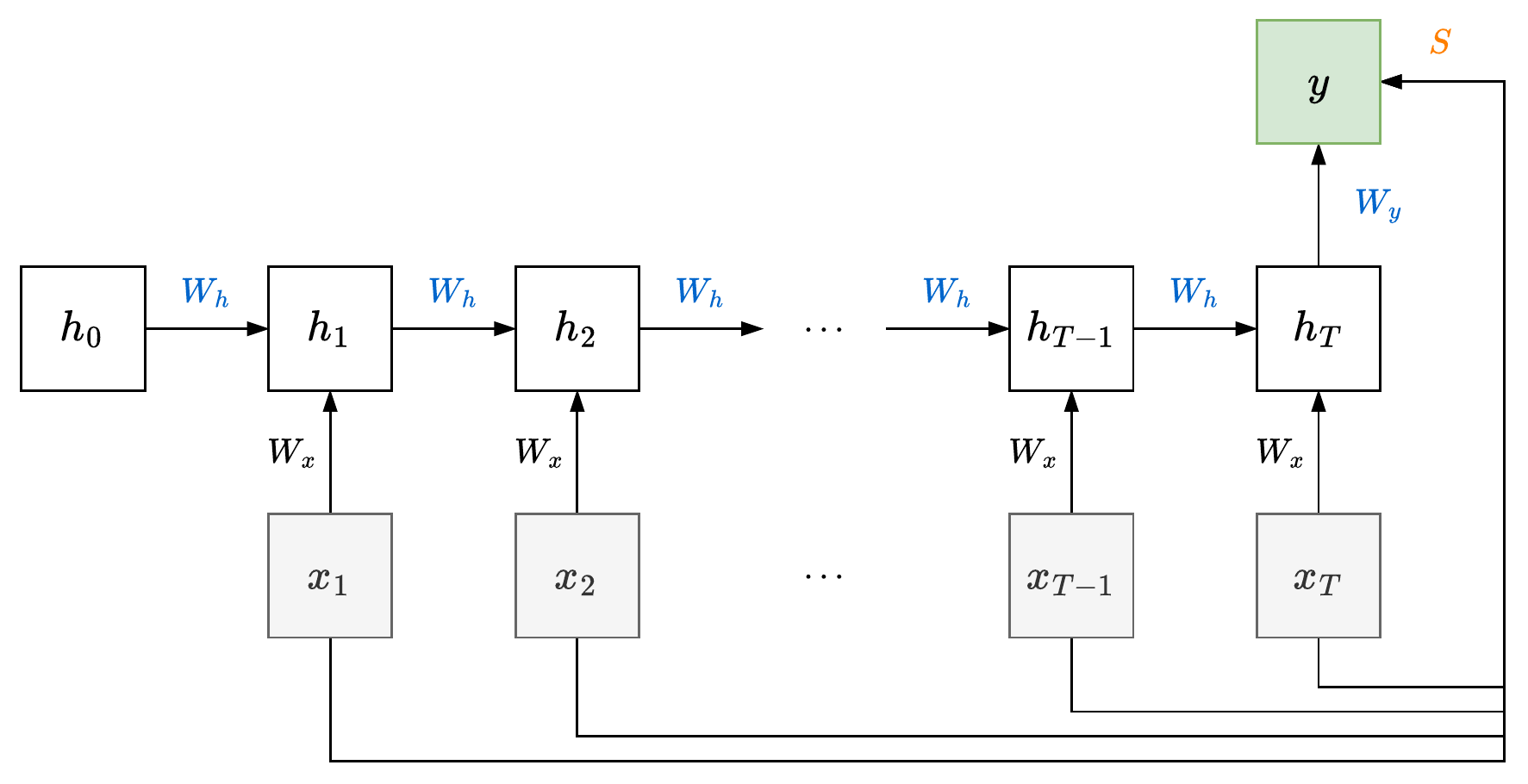}
    \caption{Architecture of a $T$-step input convex recurrent neural network (ICRNN).}
    \label{fig:icrnn}
\end{figure}

Our architecture also differs from the original ICRNN of~\citet{chen_optimal_2019} in three aspects. 
First, their many-to-many structure produces an output at every time step as required for model predictive control, whereas our many-to-one variant emits a single output after the final step, matching the scalar quantity we seek to predict. 
Second, we retain only the connections essential for convexity, discarding those from the previous input to the hidden state and from the previous hidden state to the output. Meanwhile, the skip connection $Sx$ generalizes the direct input--output connection and grants the output direct affine dependence on all time steps. 
Third, since the sign constraints apply only to $W_h$ and $W_y$, the matrices $W_x$ and $S$ acting directly on the input remain unrestricted, removing the need for the input duplication $\hat{x} = [x^\top, -x^\top]^\top$ used in their design. 
Taken together, these choices yield a streamlined ICRNN tailored to many-to-one prediction, one that preserves input convexity with fewer parameters and translates into a more compact LP reformulation.

The LP reformulation of ICRNN models follows the same ReLU epigraph mechanism as in the FICNN case. When the recurrent structure is unrolled over time, it forms a layered computational graph with non-negative downstream weights, allowing the tightness result in Proposition~\ref{prop:icnn_lp_tightness} to apply analogously. Unlike FICNNs, whose reformulation complexity scales with the hidden layer size and remains independent of input dimensions, the number of variables and constraints resulting from embedding ICRNN models scales with both the hidden state size and the input sequence length. Notably, the use of ReLU in ICRNN, required here for convexity and LP reformulation, departs from the standard practice of using tanh or sigmoid functions in RNNs. This exacerbates the risk of exploding gradients during backpropagation through time and also introduces the risk of neuron death~\citep{nwankpa_activation_2018}. To mitigate these risks, weight initialization with the identity matrix~\citep{le_simple_2015} and gradient clipping~\citep{pascanu_on_2013} are applied during the training process. 

\subsubsection{Partially input convex neural network}

Compared to standard FICNN, PICNN models may exhibit enhanced representation capacity by accommodating two distinct groups of inputs: the decision variables~$x$ of the original optimization problem and the auxiliary features~$y$. The outputs of PICNN models are convex only in decision variables, while no convexity requirements are imposed on the remaining inputs, thereby affording greater modeling flexibility than FICNN models. In this work, we consider two PICNN variants that differ in how the auxiliary feature stream derived from~$y$ interacts with the input-convex stream associated with~$x$. PICNN Type~I follows the original gated architecture of~\citet{amos_input_2017}, whereas PICNN Type~II is a simplified version proposed in this work.

As shown in Figure~\ref{fig:picnn_v1}, PICNN Type I utilizes a gating mechanism to modulate the decision variables stream with features extracted from the auxiliary inputs and is defined as follows:
\begin{subequations} \label{eq:picnn_v1}
\begin{align}
    v^{(l)} &= \sigma_1(W_v^{(l)} v^{(l-1)} + b_v^{(l)}), && l = 1,\ldots,L-1 \label{eq:picnn_v1:v_stream}\\
    h^{(1)} &= \sigma_2(S^{(1)} (x \odot \tilde{x}^{(1)}) + W_{vh}^{(1)} v^{(0)} + b_h^{(1)}), \label{eq:picnn_v1:h_first}\\
    h^{(l)} &= \sigma_2\big(W_h^{(l)} (h^{(l-1)} \odot \tilde{h}^{(l)}) + S^{(l)} (x \odot \tilde{x}^{(l)} ) + W_{vh}^{(l)} v^{(l-1)} + b_h^{(l)}\big),  && l=2,\ldots,L-1 \label{eq:picnn_v1:h_hidden}\\
    h^{(L)} &= W_h^{(L)} h^{(L-1)} + S^{(L)} h^{(0)} + W_{vh}^{(L)} v^{(L-1)} + b_{h}^{(L)}, \label{eq:picnn_v1:h_output}\\
    \tilde{h}^{(l)} &= \sigma_3(W_{v \tilde{h}}^{(l)} v^{(l-1)} + b_{v \tilde{h}}^{(l)}), && l=1,\ldots, L-1 \label{eq:picnn_v1:h_gate}\\
    \tilde{x}^{(l)} &= W_{v \tilde{x}}^{(l)} v^{(l-1)} + b_{v \tilde{x}}^{(l)}, && l=1,\ldots, L-1, \label{eq:picnn_v1:x_gate}
\end{align}
\end{subequations}
where~$\odot$ denotes the Hadamard product, namely the element-wise product. 

\begin{figure}[tb]
    \centering
    \includegraphics[width=1.00\linewidth]{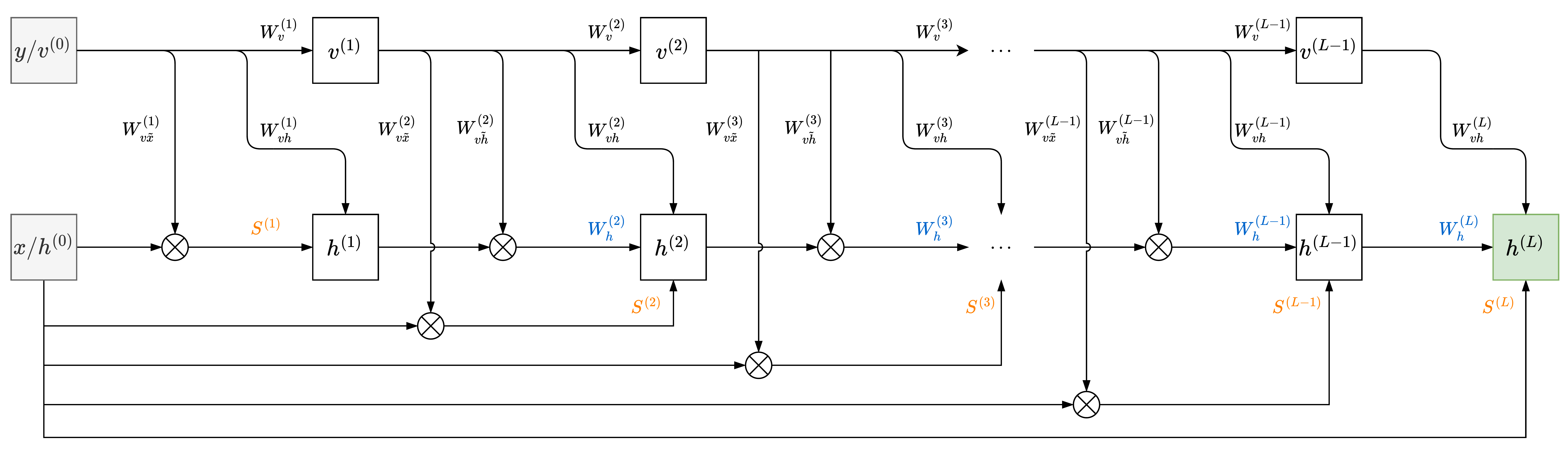}
    \caption{Architecture of an $L$-layer partially input convex neural network (PICNN Type I).}
    \label{fig:picnn_v1}
\end{figure}

The proposed variant, PICNN Type~II, is obtained by simplifying Type I. Specifically, the multiplicative modulation terms~$\tilde{x}^{(l)}$ and~$\tilde{h}^{(l)}$ and their associated gating transformations are removed. As illustrated in Figure~\ref{fig:picnn_v2}, the auxiliary stream is instead coupled to the input-convex stream through direct additive terms. The resulting architecture is defined as follows:
\begin{subequations} \label{eq:picnn_v2}
\begin{align}
    v^{(l)} & = \sigma_1(W_v^{(l)} v^{(l-1)} + b_v^{(l)}), \quad & l = 1,\ldots,L-1 \label{eq:picnn_v2:v_stream}\\
    h^{(1)} & = \sigma_2(S^{(1)} h^{(0)} + W_{vh}^{(1)} v^{(0)} + b_h^{(1)}), \label{eq:picnn_v2:h_first}\\
    h^{(l)} & = \sigma_2(W_h^{(l)} h^{(l-1)} + S^{(l)} h^{(0)} + W_{vh}^{(l)} v^{(l-1)} + b_h^{(l)}), \quad & l=2,\ldots,L-1 \label{eq:picnn_v2:h_hidden}\\
    h^{(L)} & = W_h^{(L)} h^{(L-1)} + S^{(L)} h^{(0)} + W_{vh}^{(L)} v^{(L-1)} + b_h^{(L)}. \label{eq:picnn_v2:h_output}
\end{align}
\end{subequations}

\begin{figure}[tb]
    \centering
    \includegraphics[width=0.80\linewidth]{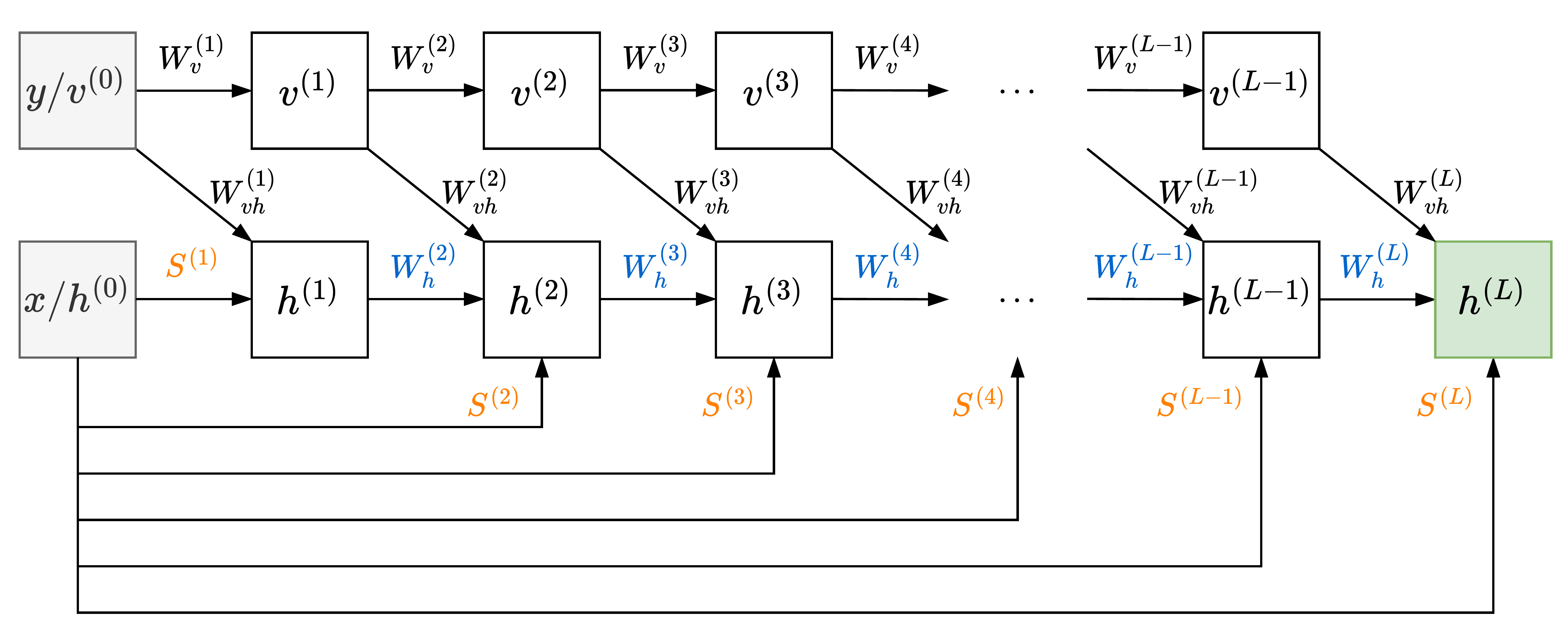}
    \caption{Architecture of an $L$-layer partially input convex neural network (PICNN Type II).}
    \label{fig:picnn_v2}
\end{figure}

In both architectures, $v^{(l)}$, defined in Eqs.~\eqref{eq:picnn_v1:v_stream} and~\eqref{eq:picnn_v2:v_stream}, represents the auxiliary feature stream generated from~$y$. The matrices~$S^{(l)}$ play the same role as the skip-connection weights in FICNNs by directly linking the decision variables to each hidden layer of the convex stream. In the gated architecture, the auxiliary stream~$v^{(l)}$ generates the modulation terms~$\tilde{x}^{(l)}$ and~$\tilde{h}^{(l)}$ through Eqs.~\eqref{eq:picnn_v1:h_gate} and~\eqref{eq:picnn_v1:x_gate}. These terms enter the convex stream through Hadamard products with~$x$ and~$h^{(l-1)}$, allowing the auxiliary features to modulate the contributions of the decision variables and the previous convex hidden layer. By removing these gating pathways, the proposed Type II architecture offers a simpler and more parsimonious alternative, whose practical benefits and trade-offs are examined in the experiments in Section~\ref{sec:experiments}.

Similar to FICNNs, the convexity of PICNN models with respect to~$x$ requires the activation function~$\sigma_2$ in Eqs.~\eqref{eq:picnn_v1} and~\eqref{eq:picnn_v2} to be convex and non-decreasing, and the weights~$W_h$ to be non-negative. For the gated architecture in Eq.~\eqref{eq:picnn_v1}, the output of~$\sigma_3$ must also be non-negative. The remaining activation functions and weights are not subject to these restrictions. This flexibility in the auxiliary pathway holds provided the auxiliary features~$y$ are fixed inputs rather than decision variables, so that the auxiliary stream and the modulation terms can be precomputed by a forward pass and treated as constants in the reformulation. Under these conditions, the convex pathway has the same layered epigraph structure and non-negative downstream weights as the FICNN, allowing the tightness result in Proposition~\ref{prop:icnn_lp_tightness} to apply analogously. Therefore, the resulting reformulation does not introduce additional constraints or decision variables beyond those required for an FICNN with the corresponding convex-path dimensions, thereby preserving comparable computational complexity.

\subsection{Surrogate models for trajectory planning} \label{sec:surrogate}

The reduced trajectory-planning problem in Eq.~\eqref{eq:reduced_original_planning} contains two objective components that are intractable for online solution. The safety surrogate is designed to replace the computationally expensive interaction-based safety cost, while the lateral comfort surrogate is designed to replace the implicit optimization required to compute the duration-dependent optimal value~$C_{\mathrm{lat}}^{\star}(\delta)$. Both surrogates utilize the input convex architectures introduced above, ensuring convexity with respect to the decision variables and enabling LP representations in the final planning formulation.

\subsubsection{Structured ICNN-class safety surrogate} 
The safety cost depends jointly on fixed historical context, planned LC trajectory, and duration. We therefore construct the architecture shown in Figure~\ref{fig:surrogate_1}, which comprises three principal modules:

\textrm{(i)} Historical encoder (non-convex branch): This module encodes the static and dynamic historical context of the LC environment. It takes the longitudinal historical trajectories of all vehicles and their respective vehicle lengths as input. Temporal features are extracted through a Long Short-Term Memory (LSTM)~\citep{hochreiter_long_1997}, while static features are processed by an MLP. The resulting representations are concatenated across all vehicles to form a context feature vector, serving as the non-convex input to the decoder.

\textrm{(ii)} Trajectory encoders (convex branch): This branch encodes the decision variables of the trajectory planning optimization problem. The planned ego LC trajectory, represented as a sequence of longitudinal positions or speeds over~$T$, is encoded by an ICRNN, while the LC duration~$\delta$ is embedded by an FICNN. The two resulting representations are concatenated to form the convex input vector, preserving convexity with respect to the original decision variables.

\textrm{(iii)} PICNN decoder: The decoder integrates the context and decision variable representations to predict the safety metric~$\hat{C}_{\text{safety}}$. By enforcing the PICNN structural constraints, the predicted output is convex with respect to the decision variables while imposing no convexity requirement on the contextual features. Skip connections of the PICNN decoder are deliberately removed. This is because, with the concatenated output of the ICRNN and FICNN serving as the PICNN's convex inputs, the validity of Proposition~\ref{prop:icnn_lp_tightness} requires non-negativity of all weights downstream of these inputs. This condition is violated by the unconstrained skip connection weights, so the validity of the LP epigraph relaxation no longer holds. This design retains sufficient expressive power, as supported by the safety-cost prediction results reported in Section~\ref{sec:prediction}, and enables valid LP reformulation, striking a balance between the representation capacity and computational efficiency.

\begin{figure}[tb]
    \centering
    \includegraphics[width=0.98\linewidth]{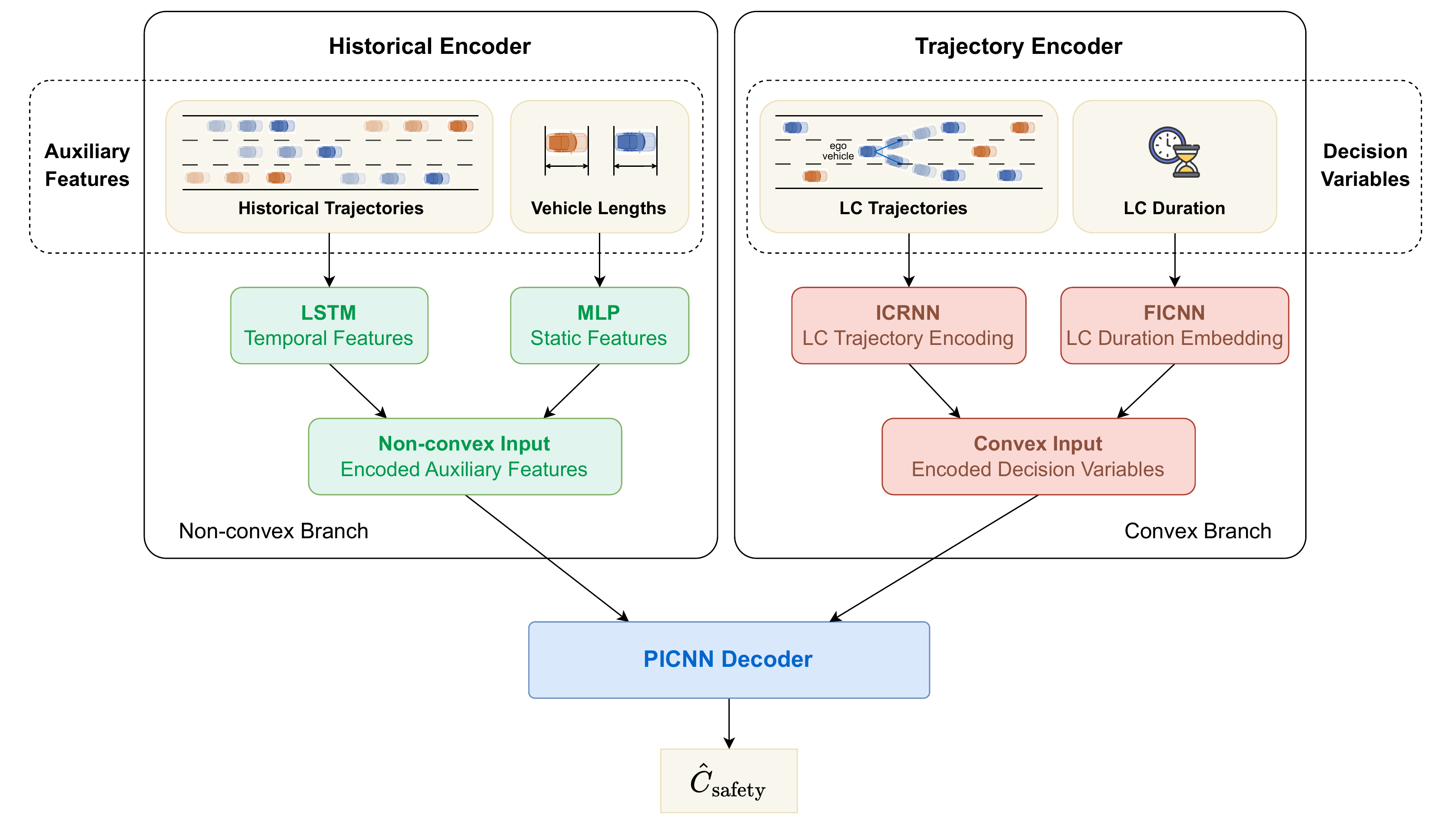}
    \caption{Architecture of the safety surrogate model, comprising a non-convex historical encoder, a convex trajectory encoder, and a skip-connection-free PICNN decoder.}
    \label{fig:surrogate_1}
\end{figure}

\subsubsection{FICNN lateral-comfort surrogate}

The optimal lateral comfort surrogate takes the LC duration~$\delta$ as input and approximates the optimal value function~$C_{\mathrm{lat}}^{\star}(\delta)$ defined by the duration-dependent optimization problem in Eq.~\eqref{eq:lateral_value_function}. This mapping adopts the FICNN architecture, which can be incorporated into the final planning problem through binary-free reformulation constraints.
As discussed in Section~\ref{sec:objective_components}, the optimal lateral comfort cost scales proportionally to~$\frac{1}{\delta^3}$ when minimizing the sum of squared lateral accelerations. Since $\delta > 0$, this relationship is convex in~$\delta$, making the FICNN a natural choice for approximating the corresponding cost function. 

\subsubsection{Output activation and training considerations} \label{sec:surrogate_considerations}

Although regression models typically use the identity activation function in the output layer, replacing it with a non-negative activation function can be beneficial and rational when the surrogate target is inherently non-negative. For the PICNN decoder, since the surrogate target $\hat{C}_{\text{safety}}$ as shown in Eq.~\eqref{eq:safety} is non-negative, the Softplus function~\citep{dugas2000incorporating}, a smooth version of ReLU, is adopted as the output layer activation. This function enforces the non-negative output range consistent with its physical meaning while mitigating training issues such as dead neurons and vanishing gradients~\citep{zheng_improving_2015}. The trade-off is that Softplus introduces nonlinear equality constraints during reformulation, rendering the reformulated problem non-convex and computationally challenging. 

To address this, the equality constraints representing the output layer forward pass are relaxed to inequality constraints, formulated as follows:
\begin{equation} \label{eq:softplus_ineq}
    \hat{f}(x) \geq \ln\left(1 + \exp\left(W^{(L)} h^{(L-1)} + S^{(L)} h^{(0)} + b^{(L)}\right)\right).
\end{equation}
By relaxing the output layer activation as in Eq.~\eqref{eq:softplus_ineq}, the feasible region is enlarged to the epigraph of the original Softplus function, and, since the Softplus function is inherently convex, its epigraph constitutes a convex set, resulting in a tractable convex relaxation. By a trivial optimality argument, Eq.~\eqref{eq:softplus_ineq} is binding at optimality, thereby preserving the convexity of the reformulated problem and guaranteeing global optimality. 

In contrast, although the lateral comfort cost is also non-negative, the comfort surrogate retains the identity output activation. Given the simplicity of this low-dimensional learning task, non-negativity need not be architecturally enforced, and the identity output keeps the corresponding reformulation fully linear.

When training ICNN-class models, the non-negativity constraints on the relevant weights are enforced by projecting weights onto the feasible region after each gradient step, according to~\citep{chen_data-driven_2020}:
\begin{equation} \label{eq:project}
    \Pi _X(y)=\arg\min_{x \in X}  \frac{1}{2}\left\|  x-y \right\| _{2}^{2}, 
\end{equation}
in which, $X$ and $y$ denote the feasible region and the weight value prior to projection, respectively. For weights subject to non-negativity constraints, this projection is implemented by clamping negative values to zero after each parameter update.

\subsection{Surrogate-embedded trajectory planning formulation}
\label{sec:final_formulation}

Let $\mathcal{E}_{\mathrm{safety}}(\xi)$ and $\mathcal{E}_{\mathrm{lat}}$ denote the binary-free convex epigraph representations of the trained safety and lateral comfort surrogates, respectively. The internal variables and constraints associated with these representations are suppressed here for notational simplicity. Replacing the two intractable objective components in Eq.~\eqref{eq:reduced_original_planning} with their surrogate representations yields
\begin{subequations}\label{eq:final_planning}
\begin{align}
    \min_{P^{\mathrm{x}},\delta}\quad
    & w_{\mathrm{s}}\hat C_{\mathrm{safety}}
      +w_{\mathrm{c}}\left[
      C_{\mathrm{lon}}(P^{\mathrm{x}})
      +\hat C_{\mathrm{lat}}^{\star}\right]
      +w_{\mathrm{e}}C_{\mathrm{eff}}(P^{\mathrm{x}})
      \label{eq:final_planning:obj}\\
    \mathrm{s.t.}\quad
    & P^{\mathrm{x}}\in\mathcal{P}_{\mathrm{x}},
      \qquad \delta_{\min}\leq\delta\leq T,
      \label{eq:final_planning:trajectory}\\
    & (P^{\mathrm{x}},\delta,\hat C_{\mathrm{safety}})
      \in\mathcal{E}_{\mathrm{safety}}(\xi),
      \label{eq:final_planning:safety}\\
    & (\delta,\hat C_{\mathrm{lat}}^{\star})
      \in\mathcal{E}_{\mathrm{lat}}.
      \label{eq:final_planning:lateral}
\end{align}
\end{subequations}
Here, the traffic context~$\xi$ is fixed during each planning instance. Compared with Eq.~\eqref{eq:reduced_original_planning}, only the safety cost and optimal lateral comfort cost are replaced by learned surrogates. The longitudinal comfort cost, efficiency cost, and physical trajectory constraints retain their analytical forms. Because the retained analytical terms are affine or convex quadratic and both surrogate epigraph sets are convex, problem~\eqref{eq:final_planning} is a binary-free convex optimization problem. Under the conditions established above, Proposition~\ref{prop:icnn_lp_tightness} ensures that the epigraph representations are tight at an optimum, while Remark~\ref{rem:objective_extension} shows that this result is preserved in the presence of the retained analytical terms and multiple surrogate outputs. Consequently, there exists an optimal solution at which both surrogate outputs coincide with their corresponding forward pass values.

\section{Numerical experiments} \label{sec:experiments}

This section evaluates the proposed methodology in terms of surrogate prediction accuracy, reformulation efficiency, and closed-loop trajectory planning performance. Section~\ref{sec:dataset} first describes the hybrid dataset used in the experiments. Section~\ref{sec:prediction} compares alternative architectures for safety-cost prediction and separately evaluates the FICNN surrogate for optimal lateral comfort cost. Section~\ref{sec:reformulation} compares the LP and MIP reformulation methods in terms of problem size and solution efficiency. Section~\ref{sec:performance} evaluates four surrogate-embedded trajectory planners against two conventional baselines by executing planned trajectories in SUMO~\citep{lopez_microscopic_2018} and measuring their realized performance.

All training experiments are conducted on the Triton computing cluster maintained by the Aalto Science-IT project, utilizing an Intel Xeon Platinum 8468 CPU and an NVIDIA H100 GPU with 80 GB of RAM. Reformulation experiments are conducted on a computing platform equipped with an AMD Ryzen 7 6800H processor and 16 GB RAM. The framework is implemented in Python 3.12.4, utilizing PyTorch 2.9.1~\citep{paszke_pytorch_2019} for model training, Pyomo 6.9.5~\citep{hart_pyomo_2011} for optimization modeling, and Gurobi 13.0.1~\citep{gurobillc_gurobi_2025} with default parameters for solving. 

\subsection{Dataset overview} \label{sec:dataset}

We adopt a hybrid dataset that combines real-world LC trajectories with simulation-generated scenarios to improve training data coverage. This design is motivated by the trust-region issue in optimization problems with embedded predictive NN models, where the optimal solutions may lie outside the range of the training dataset. In such cases, surrogate predictions and the resulting decisions become unreliable, since solvers can exploit extrapolation errors. The simulated scenarios therefore complement the real-world trajectory data by broadening the range of the dataset and enhancing robustness against rare situations.

The adopted dataset comprises 1,317 real-world scenarios sourced from the SWIFTraj dataset~\citep{han_swiftraj_2026} and 2,901 augmented scenarios generated via SUMO simulation. Each scenario consists of four vehicles: the ego vehicle, the obstacle vehicle, and the leading vehicles in both the initial and target lanes. For each vehicle, longitudinal positions and speeds are recorded at $10$ Hz over the $3$ seconds interval preceding the LC maneuver, yielding a 30-step historical trajectory per vehicle. The LC trajectory of the ego vehicle is discretized at $1$ second intervals over the fixed horizon $T$, thereby producing a 10-step trajectory profile and the associated LC maneuver duration $\delta$. The TET-based and TIT-based safety metrics of each scenario are calculated following Eqs.~\eqref{eq:tet} and~\eqref{eq:tit}, with the $\mathrm{TTC}^\star$ set as $1.5$~\citep{wang_review_2021}.

SWIFTraj data were collected from a highway in Nanjing, China, depicted in Figure~\ref{fig:highway}, during the morning peak period on workdays. The dataset contains detailed information for each LC vehicle, including lateral and longitudinal position, speed, acceleration, and the motion states of surrounding vehicles. Furthermore, since this dataset was collected during the morning peak hour, it comprehensively records the LC trajectories and vehicle interactions under congested conditions. 
Because simulation platforms such as SUMO rely on parameterized car-following and LC models, they may not fully reproduce congested LC interactions without careful calibration. The real-world component of the hybrid dataset is therefore essential for capturing high-density driving behavior. This distinction is clearly reflected in Figure~\ref{fig:mean_v}, where the mean speed derived from the real-world trajectory data is significantly lower than that of the simulation data.

\begin{figure}[tb]
    \centering
    \includegraphics[width=0.80\linewidth]{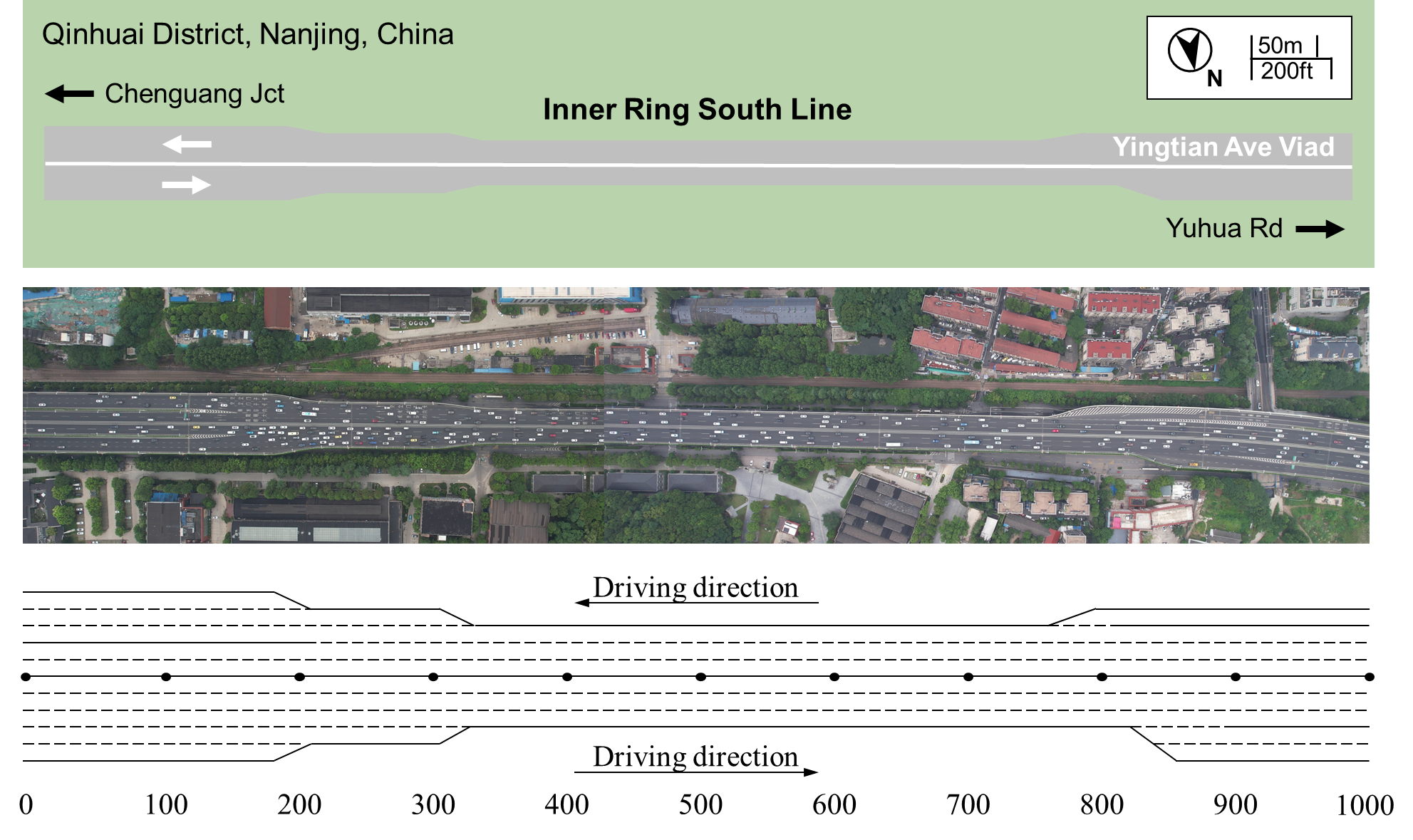}
    \caption{The geographical map of the site for SWIFTraj data collection, reproduced from~\citet{han_swiftrajweb_2026}.}
    \label{fig:highway}
\end{figure}

\begin{figure}[tb]
    \centering
    \includegraphics[width=0.65\linewidth]{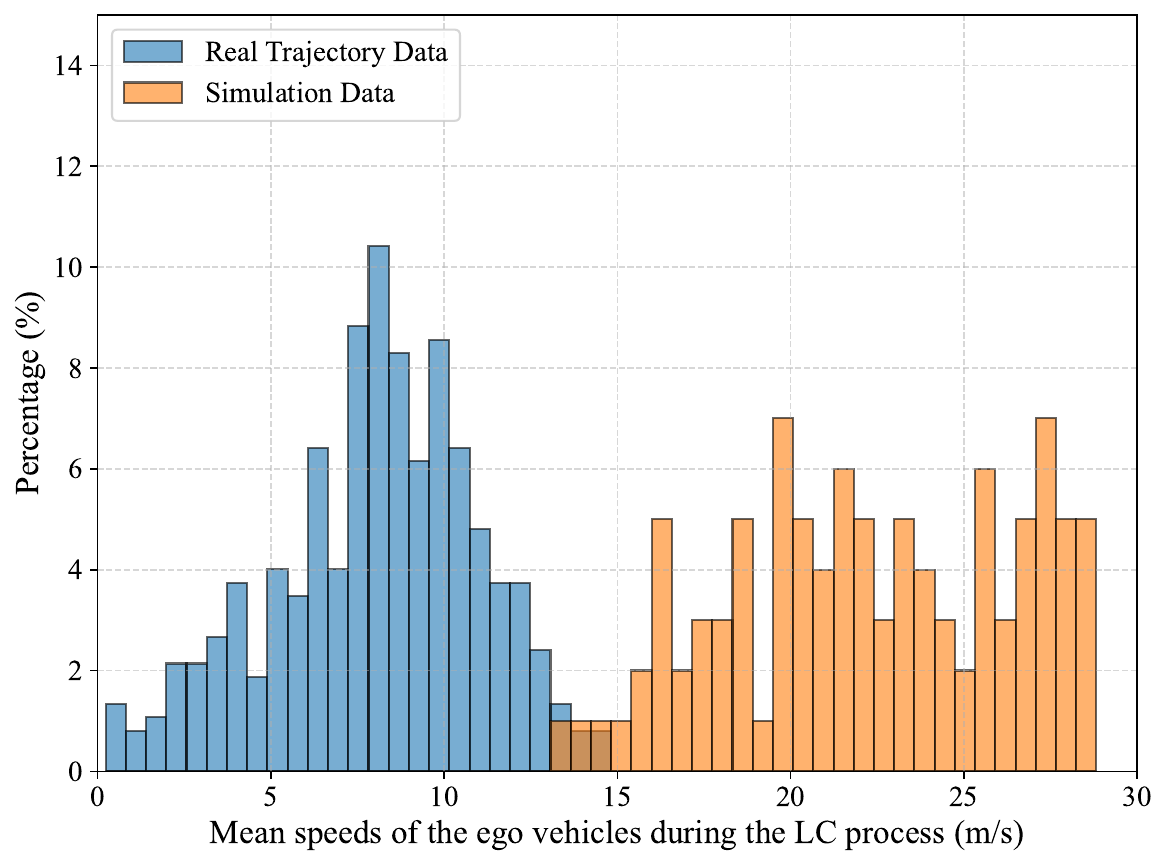}
    \caption{Distribution of mean speeds of ego vehicles during the LC process, comparing real trajectory data and simulation data. }
    \label{fig:mean_v}
\end{figure}

In the SUMO simulation framework, the default car-following and LC models are applied for all vehicles, while the associated parameters are systematically varied~\citep{erdmann2015sumo, lopez_microscopic_2018}. Furthermore, high-risk scenarios are generated in SUMO, thereby covering situations that are rare or absent from the real-world dataset.
As shown in Figs.~\ref{fig:tet} and~\ref{fig:tit}, the partial TET-based and TIT-based safety metric values are overwhelmingly larger than any values observed in the real-world dataset. These elevated values, which represent high-risk scenarios, provide adequate data diversity for the subsequent surrogate model training. Additionally, surrogate models face a greater challenge when trained on TIT-based safety metrics compared to TET-based ones. This discrepancy is primarily attributed to the wider numerical range and the presence of extreme values in the TIT-based dataset. These characteristics may influence the prediction performance across different model architectures, which will be presented and discussed in subsequent sections.

\begin{figure}[tb]
    \centering
    \includegraphics[width=0.65\linewidth]{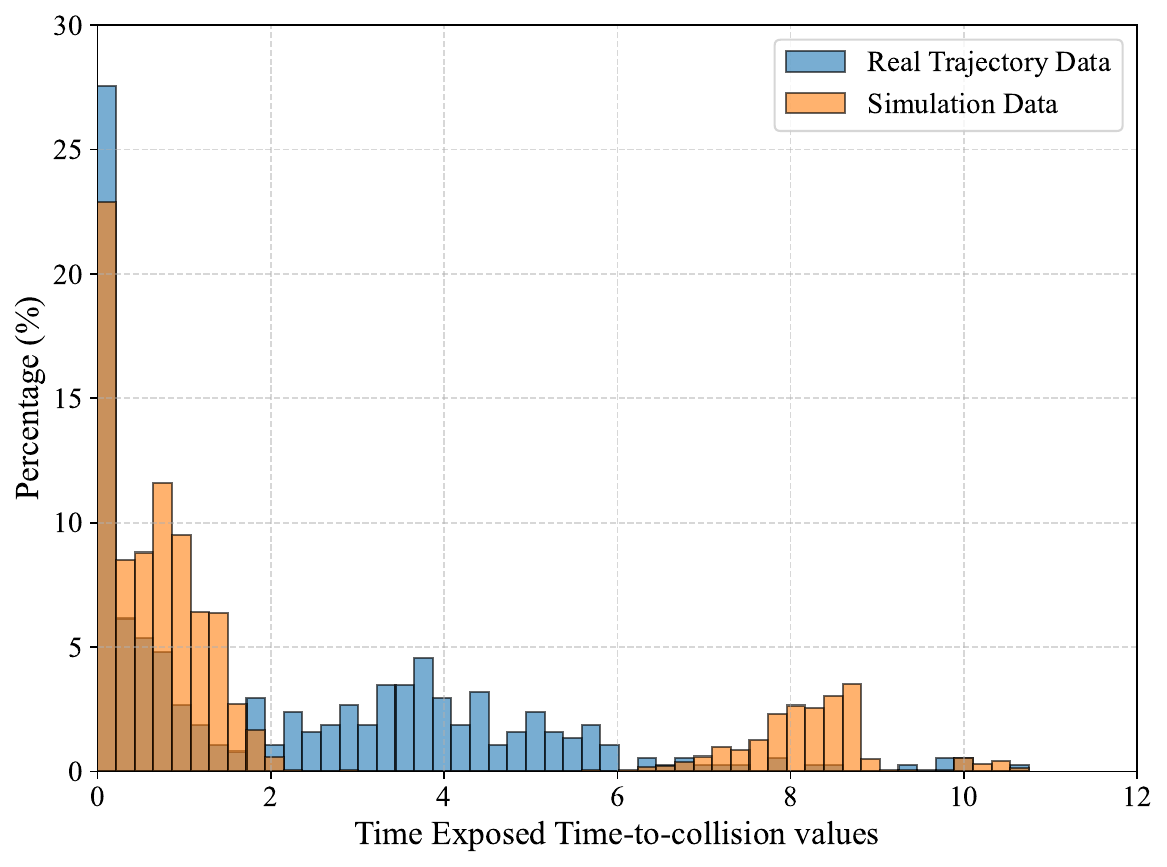}
    \caption{Distribution of TET-based safety metrics comparing real trajectory data and simulation data}
    \label{fig:tet}
\end{figure}

\begin{figure}[tb]
    \centering
    \includegraphics[width=0.65\linewidth]{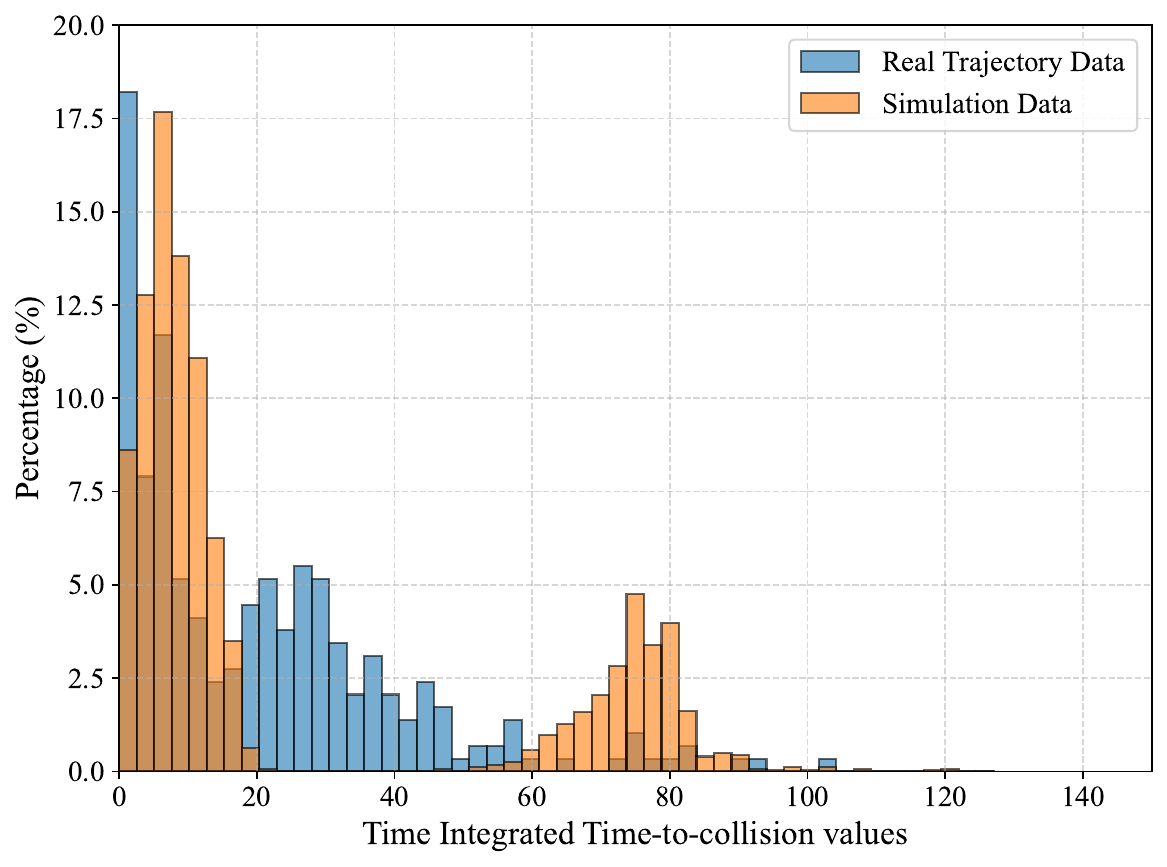}
    \caption{Distribution of TIT-based safety metrics comparing real trajectory data and simulation data}
    \label{fig:tit}
\end{figure}

\subsection{Model prediction performance} \label{sec:prediction} 

While the representation capacity of ICNN-class models is theoretically limited by their non-negative weight constraints, the practical implications of this constraint require empirical evaluation. This section, therefore, assesses the prediction performance of ICNN-class surrogate models against their unconstrained counterparts. The evaluation covers the two surrogate models described in Section~\ref{sec:surrogate}, which predict the total safety cost and lateral comfort cost during the entire LC process, respectively.

The proposed surrogate model for safety metrics incorporates static features and dynamic vehicle states as inputs, as shown in Figure~\ref{fig:surrogate_1}. Static information includes the lengths of all involved vehicles, while dynamic information comprises the historical speed and position profiles of all vehicles and the ego vehicle's trajectories during the LC maneuver. These profiles are normalized relative to the ego vehicle, thereby facilitating model training. Moreover, this relative representation reduces the input range, enabling smaller big-$M$ values in the reformulation. The tightening of the search space significantly improves computational efficiency when solving optimization problems integrating with trained models.

Table~\ref{tab:compare model} summarizes the four architectures evaluated for safety-cost prediction. All models share the same processing pipeline for static features and historical state information, but differ in encoders for LC trajectories and duration. The Input Reduced Baseline omits these planned-trajectory inputs to assess their contribution, whereas the unconstrained Baseline uses LSTM and MLP modules to encode the trajectory and duration. We refer to the two ICNN-class safety-surrogate configurations as Proposed Model I and Proposed Model II, which employ the gated PICNN Type~I decoder and the simplified additive PICNN Type~II decoder, respectively. Both configurations share the same historical encoder, ICRNN trajectory encoder, and FICNN duration encoder, differing only in their PICNN decoder architectures. For each of the four architectures, separate model instances are trained to predict TET-based and TIT-based safety costs.

\begin{table}[tb]
    \footnotesize
    \centering
    \caption{Architectural comparison of the evaluated surrogate models.}
    \label{tab:compare model}
    \begin{tabular}{lccc}
        \toprule
        Model & \makecell{LC Trajectories \\ Encoder} & \makecell{LC Duration \\ Encoder} & Decoder\\
        \midrule
        Input Reduced Baseline & -- & -- & MLP \\
        Baseline               & LSTM  & MLP   & MLP \\
        Proposed Model I       & ICRNN & FICNN & PICNN Type I \\
        Proposed Model II      & ICRNN & FICNN & PICNN Type II \\
        \bottomrule
    \end{tabular}
\end{table}

The MLP decoder in both the Input Reduced Baseline and Baseline models consists of three hidden layers with dimensions [64, 64, 32]. For the Baseline model, the LSTM encoder for LC trajectories has a hidden dimension of 16, and the MLP encoder for LC duration has hidden dimensions of [8, 8]. For Proposed Models I and II, the ICRNN encoder for LC trajectories has hidden and output dimensions of 8 and 16, respectively, and the FICNN encoder for LC duration shares the same hidden dimensions as the MLP duration encoder. Although PICNN Types~I and~II differ in architecture, both decoders maintain the same hidden dimensions: [32, 32, 16] for the convex path and [64, 64, 32] for the non-convex path. ReLU activation functions are applied between all hidden layers across all models. The output layer activations follow the selections described in Section~\ref{sec:surrogate_considerations}, with Softplus for the PICNN decoder in the safety surrogate and identity elsewhere.

All models are trained using the Adam optimizer~\citep{kingma_adam_2017} with an initial learning rate of $10^{-3}$, a batch size of 32, and a maximum of 300 epochs. A ReduceLROnPlateau learning rate scheduler~\citep{PyTorch_reducelronplateau_2025} with a patience of 10 and a decay factor of 0.5 is applied. The hybrid dataset described in Section~\ref{sec:dataset} is randomly partitioned into training, validation, and test sets at a ratio of $70{:}15{:}15$. Furthermore, training PICNN models is particularly challenging due to their special architectures and non-negative weight constraints, so the selection of an appropriate initialization strategy is critical. Specifically, the weights for skip and cross-connections should be initialized to near-zero values to ensure a more stable training process.

Tables~\ref{tab:tet} and~\ref{tab:tit} present the safety-cost prediction accuracy of the four architectures for TET-based and TIT-based safety costs, respectively. The reported statistics are obtained from 20 independent training runs using the same hyperparameters. The Input Reduced Baseline model exhibits the lowest prediction accuracy for both targets, indicating that the LC trajectories and duration of the ego vehicle contain substantial information for safety evaluation and cannot be omitted. Although the unconstrained Baseline achieves the highest accuracy, Proposed Models I and II maintain competitive performance while satisfying the structural requirements for convex embedding. For both prediction targets, Proposed Model I consistently outperforms Proposed Model II across both labels, suggesting that the gating mechanism of the original PICNN provides greater representational flexibility than the simplified additive decoder.

\begin{table}[tb]
    \footnotesize
    \centering
    \caption{Prediction performance of surrogate models for TET-based safety metric evaluated on the test dataset. Values are presented as mean $\pm$ standard deviation.}
    \label{tab:tet}
    \begin{tabular}{lccc}
        \toprule
        Model & MAE & RMSE & $R^2$ \\
        \midrule
        Input Reduced Baseline & $0.6055 \pm 0.0067$ & $1.3785 \pm 0.0061$ & $0.8072 \pm 0.0017$ \\
        Baseline               & $0.1525 \pm 0.0236$ & $0.4580 \pm 0.0342$ & $0.9786 \pm 0.0031$ \\
        Proposed Model I       & $0.2865 \pm 0.0072$ & $0.7026 \pm 0.0070$ & $0.9499 \pm 0.0010$ \\
        Proposed Model II      & $0.3225 \pm 0.0055$ & $0.7626 \pm 0.0073$ & $0.9410 \pm 0.0011$ \\
        \bottomrule
    \end{tabular}
\end{table}

\begin{table}[tb]
    \footnotesize
    \centering
    \caption{Prediction performance of surrogate models for TIT-based safety metric evaluated on the test dataset. Values are presented as mean $\pm$ standard deviation.}
    \label{tab:tit}
    \begin{tabular}{lccc}
        \toprule
        Model & MAE & RMSE & $R^2$ \\
        \midrule
        Input Reduced Baseline & $5.7773 \pm 0.1306$ & $14.3723 \pm 0.0649$ & $0.7650 \pm 0.0021$ \\
        Baseline               & $2.1371 \pm 0.3413$ & $7.2389  \pm 1.1846$ & $0.9388 \pm 0.0196$ \\
        Proposed Model I       & $2.9469 \pm 0.0896$ & $7.5187  \pm 0.4438$ & $0.9355 \pm 0.0081$ \\
        Proposed Model II      & $3.5098 \pm 0.0882$ & $9.6971  \pm 0.2507$ & $0.8930 \pm 0.0055$ \\
        \bottomrule
    \end{tabular}
\end{table}

We next evaluate a separate surrogate for lateral-comfort cost prediction, distinct from the two safety-surrogate configurations considered above. As discussed in Section~\ref{sec:surrogate}, the optimal lateral comfort cost is convex in the LC duration~$\delta$ when the objective is to minimize the sum of squared lateral accelerations. To investigate whether the ICNN-class models exhibit better performance when the underlying input-output relationship is convex, we compare an FICNN with a standard MLP using the same hidden-layer dimensions to predict the optimal lateral-comfort cost across LC durations. Both models comprise two hidden layers with dimensions~[16, 16] and are trained using the Adam optimizer with a learning rate of $10^{-3}$, a batch size of 16, and a maximum of 300 epochs. Table~\ref{tab:lateral_comfort} reports the experiment results over 50 random seeds with the same dataset split. The FICNN achieves consistently lower MAE and RMSE than the MLP, suggesting that its convex structure effectively captures the underlying relationship between the optimal cost and the LC duration. Figure~\ref{fig:lateral_model} presents the fitted curves alongside sampled ground-truth values.

\begin{table}[tb]
    \footnotesize
    \centering
    \caption{Prediction performance of FICNN and MLP models for the optimal lateral comfort cost as a function of LC duration.}
    \label{tab:lateral_comfort}
    \begin{tabular}{lccc}
        \toprule
        Model & MAE & RMSE & $R^2$ \\
        \midrule
        MLP   & $0.0306 \pm 0.0088$ & $0.0772 \pm 0.0188$ & $0.9969 \pm 0.0016$ \\
        FICNN & $0.0150 \pm 0.0011$ & $0.0408 \pm 0.0012$ & $0.9992 \pm 0.0001$ \\
        \bottomrule
    \end{tabular}
\end{table}

\begin{figure}[tb]
    \centering
    \includegraphics[width=0.65\linewidth]{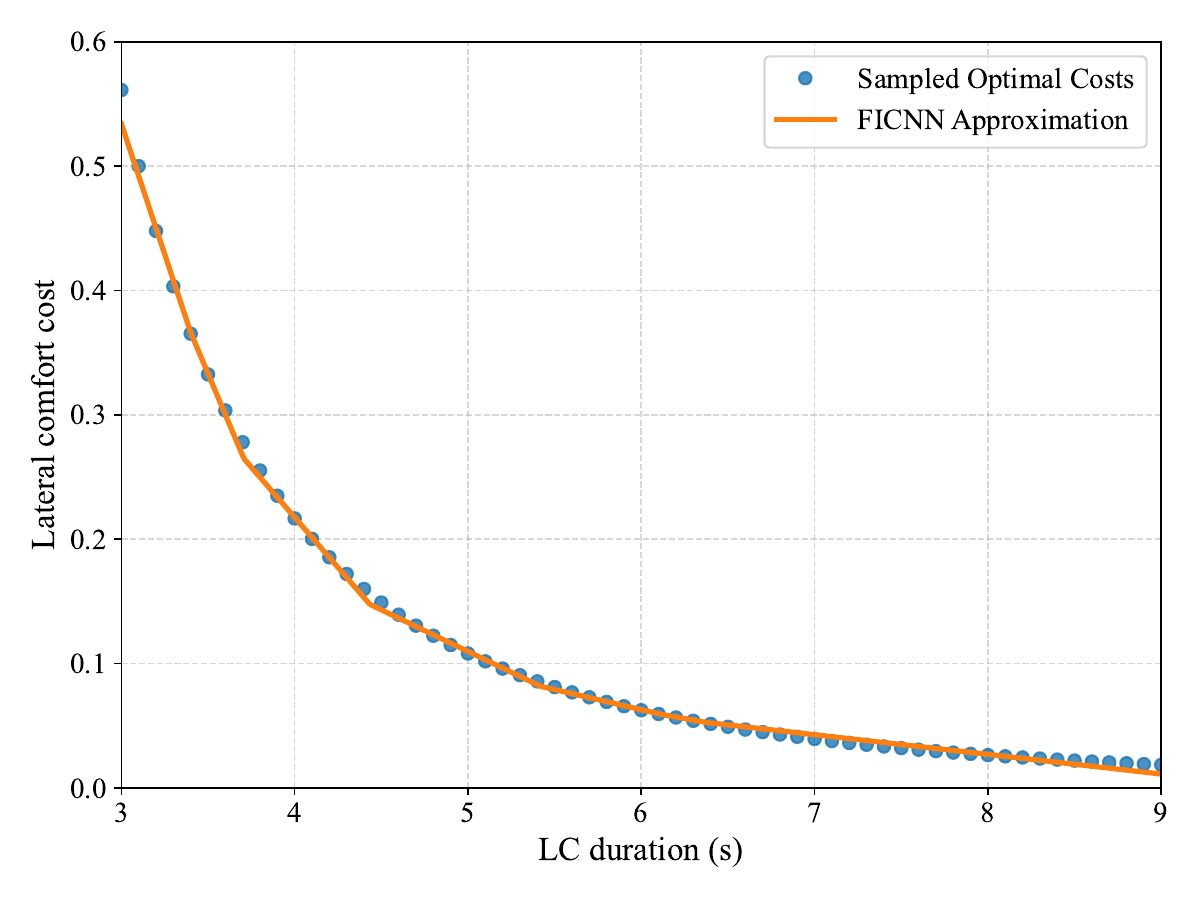}
    \caption{Predicted optimal lateral comfort cost of the FICNN surrogate against the sampled ground truth as a function of the LC duration.}
    \label{fig:lateral_model}
\end{figure}

Overall, the safety-cost experiments demonstrate that Proposed Models I and II achieve competitive predictive accuracy while satisfying the structural requirements for convex embedding, although their accuracy remains below that of the unconstrained Baseline. In contrast, the separate lateral-comfort experiment demonstrates that the FICNN outperforms the MLP when the target mapping is inherently convex. Together, these findings support the use of ICNN-class surrogates to represent the two distinct objective components within the proposed planning framework.

\subsection{Reformulation performance} \label{sec:reformulation}

This section evaluates the computational cost of embedding the two safety-surrogate architectures, Proposed Models I and II, introduced in Section~\ref{sec:prediction}. For each architecture, we compare the binary-free LP reformulation based on Eq.~\eqref{eq:icnn_lp} with a conventional MIP reformulation of the ReLU layers based on Eq.~\eqref{eq:fnn_mip}. The comparison assesses the resulting optimization problem sizes and solution times, while also verifying the consistency between the LP-based embedding and the trained surrogate. In both formulations, weights with absolute values below $10^{-4}$ are pruned to increase sparsity and improve computational efficiency. 
The Softplus output layer for the PICNN decoder of the safety surrogate is handled by the epigraph relaxation described in Section~\ref{sec:surrogate_considerations}.

The number of binary variables, continuous variables, and constraints, along with the solution time over 100 randomly selected LC scenarios, are reported in Table~\ref{tab:reformulate}. The LP reformulation introduces only continuous variables, whereas the MIP reformulation additionally introduces binary variables and extra constraints to represent the ReLU activations. Consequently, the LP reformulation is at least 30 times faster than the MIP reformulation for both model types. With an average solve time of under 0.1 seconds and better numerical stability, the LP-based approach is suitable for real-time trajectory planning. Across all 100 scenarios, the LP and MIP reformulations attain identical optimal objective values, and the surrogate output at the LP optimum coincides with the forward pass of the trained models, consistent with the tightness established in Proposition~\ref{prop:icnn_lp_tightness}.

\begin{table}[tb]
    \footnotesize
    \centering
    \caption{Computational performance and structural statistics of the two safety-surrogate architectures, Proposed Models I and II, under LP and MIP reformulations. Values are presented as mean $\pm$ standard deviation.}
    \label{tab:reformulate}
    \begin{tabular}{lccccc}
        \toprule
        Model & \makecell{Reformulation \\ Type} & \makecell{Binary \\ Variables} & \makecell{Continuous \\ Variables} & Constraints & \makecell{Solution \\ Time (s)} \\
        \midrule
        Proposed Model I  & LP  & --  & 255 & 255 & $0.0560 \pm 0.0156$ \\
        Proposed Model I  & MIP & 200 & 255 & 655 & $2.2146 \pm 1.0703$ \\
        Proposed Model II & LP  & --  & 255 & 255 & $0.0512 \pm 0.0107$ \\
        Proposed Model II & MIP & 200 & 255 & 655 & $1.5420 \pm 0.8623$ \\
        \bottomrule
    \end{tabular}
\end{table}

\subsection{Planner performance} \label{sec:performance}

As presented in Section~\ref{sec:prediction}, the  Proposed Models I and II achieve competitive prediction accuracy while satisfying the structural requirements for convex optimization. However, prediction accuracy alone does not necessarily reflect the performance of a surrogate when embedded in a trajectory planner. To evaluate the proposed framework in closed-loop operation, the optimized trajectories are executed in SUMO and compared against two conventional planning methods.

Four planner variants are constructed by combining the two safety-surrogate architectures with the two safety metrics. The I-TET and I-TIT planners embed the TET-trained and TIT-trained instances of Proposed Model I, respectively. Similarly, the II-TET and II-TIT planners embed the corresponding instances of Proposed Model II. In these abbreviations, I/II denotes the safety-surrogate architecture, whereas TET/TIT denotes both the training target and the safety-cost term used in trajectory optimization. All four planners share the same optimization formulation and lateral-comfort surrogate, differing only in the safety-surrogate architecture and safety metric.

Baseline~1 employs an optimization-based planner under the constant speed assumption, whereby surrounding vehicles are assumed to maintain their initial longitudinal speed throughout the entire LC process~\citep{Nilsson_lane_2017}. Baseline~2 adopts a simulation-based candidate trajectory planning approach with interaction-aware prediction, in which all candidates are evaluated while accounting for the responses of surrounding vehicles~\citep{evestedt2016interaction, trauth2024frenetix}. For each initial state, 25 longitudinal speed profiles and 5 lateral position profiles that satisfy all physical constraints are generated and combined, yielding 125 trajectory candidates. Vehicle interactions are simulated using the native Krauss car-following model~\citep{krauss1997metastable}, with the same parameters employed during training dataset generation.
To ensure a consistent trajectory representation across all planners, the degrees of the Bézier curves for longitudinal speed and lateral position are fixed as 7 and 9, respectively. The objective weights for TET safety, TIT safety, comfort, and efficiency are set at 50, 5, 10, and 1, reflecting the prioritization of safety in trajectory optimization. The different safety-cost weights account for the numerical scales of the two safety metrics. For performance evaluation, both TET-weighted and TIT-weighted objective values are computed for each planner, regardless of the safety metric used in its optimization objective.
Since the validity of the LP reformulation is verified in Section~\ref{sec:reformulation}, each trajectory planner incorporates its corresponding trained model through the LP reformulation to enhance computational efficiency. Finally, all baseline and proposed planners are evaluated on the identical 100 test scenarios that are excluded from the training dataset.

Table~\ref{tab:overall_performance} compares the mean TET- and TIT-weighted objective values of the six planners over the 61 common successful scenarios, in which all planners complete the LC maneuver without collisions. The numbers of infeasible and collision scenarios are reported over all 100 test scenarios. Note that an infeasible scenario does not necessarily indicate that the LC maneuver is impossible, but rather that the planner fails to find a feasible solution under its assumptions or within a fixed number of trajectory candidates. Furthermore, these limitations may also lead to collisions during trajectory execution.
As shown in Table~\ref{tab:overall_performance}, Baseline~1 exhibits the poorest performance due to the limitation of the oversimplified constant-speed assumption in capturing interactions among vehicles. In comparison, Baseline~2 performs better, benefiting from its integration of interaction-aware predictions.
Nevertheless, its planning performance remains constrained by the finite set of candidate trajectories. 
All four proposed planners successfully complete all scenarios without any infeasibility or collisions, demonstrating improved feasibility and safety.
Moreover, in terms of weighted objectives, all four proposed planners outperform Baseline~1 under both objective formulations. Compared with Baseline~2, all four proposed planners achieve lower mean TET-weighted costs, while three of them also achieve lower mean TIT-weighted costs, with I-TIT being the exception. Overall, these results demonstrate that the proposed planners can improve both planning reliability and overall objective performance by integrating interaction models into continuous trajectory optimization.

\begin{table}[tbp]
    \footnotesize
    \centering
    \caption{Comparison of six planners across 100 test scenarios. Mean weighted objective values are computed over the 61 scenarios in which all planners completed LC maneuvers without collisions. The numbers of infeasible and collision scenarios are reported over 100 test scenarios. Both weighted-cost columns are evaluated for every planner, independent of the safety metric used during optimization.}
    \label{tab:overall_performance}
    \begin{tabular}{lcccc}
        \toprule
        Planner & Weighted Cost (TET) & Weighted Cost (TIT) & Infeasible Scenarios & Collision Scenarios \\
        \midrule
        Baseline 1 & 103.475 & 86.514 & 23 & 13 \\
        Baseline 2 & 61.957 & 44.525 & 2 & 7 \\
        I-TET     & 51.141 & 44.046 & - & - \\
        I-TIT     & 60.461 & 51.078 & - & - \\
        II-TET    & 42.950 & 35.800 & - & - \\
        II-TIT    & 44.593 & 35.246 & - & - \\
        \bottomrule
    \end{tabular}
\end{table}

The individual objective components presented in Table~\ref{tab:raw_metrics} provide further insight into the performance improvement realized by the proposed planners. All four proposed planners attain lower safety costs than both baseline methods under the TET and TIT, which constitutes the primary contribution to the reduced weighted objectives. These reductions indicate that the proposed planners can more effectively mitigate risky interactions with surrounding vehicles. 
However, the improvements in comfort and efficiency are not uniform across the proposed planners. Both Model~II planners achieve substantially lower longitudinal comfort costs. The Model I planners also considerably improve longitudinal comfort relative to Baseline 1, although their costs remain slightly higher than those of Baseline 2. These differences are consistent with the longitudinal speed profiles in the first row of Figure~\ref{fig:speed_profiles}, where the proposed planners exhibit distinct patterns of speed adjustment during the LC process. Since the longitudinal comfort cost is defined as the sum of squared longitudinal accelerations in Eq.~\eqref{eq:comfort:longitudinal}, smoother speed transitions generally result in lower comfort costs. In contrast, all four proposed planners incur higher lateral comfort costs than Baseline 2, suggesting that reducing interaction-related safety risks may require longer LC time.

The efficiency costs also vary considerably among the planners. According to Eq.~\eqref{eq:efficiency}, the efficiency cost quantifies the accumulated difference between the initial speed and the planned longitudinal speed over the planning horizon. In particular, II-TIT achieves the lowest efficiency cost. This is consistent with the longitudinal speed profiles illustrated in Figure~\ref{fig:speed_profiles}, where II-TIT exhibits the most aggressive acceleration behavior in several representative scenarios. Furthermore, as illustrated in the second row of Figure~\ref{fig:speed_profiles}, the proposed planners generate similar longitudinal speed profiles in certain scenarios, despite adopting different surrogate architectures and safety metrics. 

\begin{table}[tbp]
    \footnotesize
    \centering
    \caption{Mean values of performance metrics for baseline methods and all proposed planners on common successful scenarios.}
    \label{tab:raw_metrics}
    \begin{tabular}{lccccc}
        \toprule
        Planner & $C_{\mathrm{safety}}$ (TET) & $C_{\mathrm{safety}}$ (TIT) & $C_{\mathrm{lon}}$ & $C_{\mathrm{lat}}$ & $C_{\mathrm{eff}}$ \\
        \midrule
        Baseline 1 & 1.020 & 6.813 & 4.569 & 0.650 & 0.264 \\
        Baseline 2 & 0.961 & 6.120 & 1.551 & 0.369 & $-$5.276\\
        I-TET    & 0.493 & 3.513 & 1.648 & 0.599 & 4.019 \\
        I-TIT    & 0.560 & 3.725 & 1.645 & 0.917 & 6.833 \\
        II-TET     & 0.513 & 3.697 & 1.206 & 0.608 & $-$0.829 \\
        II-TIT     & 0.602 & 4.147 & 1.220 & 0.864 & $-$6.330 \\
        \bottomrule
    \end{tabular}
\end{table}

\begin{figure}
    \centering
    \includegraphics[width=\linewidth]{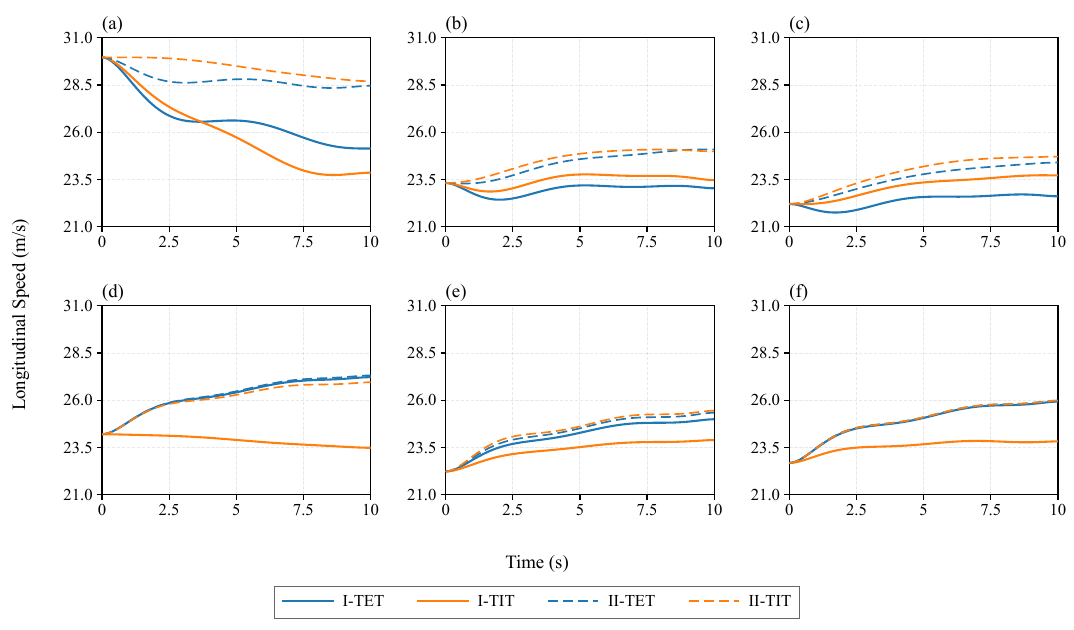}
    \caption{Longitudinal speed profiles generated by the proposed planners on representative LC scenarios.}
    \label{fig:speed_profiles}
\end{figure}

Table~\ref{tab:overall_performance} shows that the Model II planners achieve lower mean weighted objective values than the Model I planners under both evaluation metrics. The scenario-level comparison in Table~\ref{tab:best_planner} provides further evidence of this advantage. Among the four proposed planners, II-TET achieves the lowest weighted objective in 44 and 50 of the 100 scenarios under the TET- and TIT-weighted evaluations, respectively. Considering the two Model II planners together, at least one of them achieves the lowest weighted objective in 70 and 66 scenarios under the respective evaluation metrics. Furthermore, the distributions of the best-performing planners are broadly similar across the two evaluation metrics, suggesting that the relative advantage of Model II is not solely determined by the choice of evaluation metric. 
This result may appear counterintuitive, as Tables~\ref{tab:tet} and~\ref{tab:tit} indicate that Model I achieves higher overall prediction accuracy on the test dataset. However, in optimization problems with embedded NNs, prediction accuracy around the optimal solutions is often more relevant to downstream solution quality than overall accuracy on the entire test dataset~\citep{liu_simulator-based_2025}. Consistent with this view, Table~\ref{tab:error} shows that Model II yields lower prediction errors when evaluated on each planner's own optimized trajectories. This discrepancy may partly explain why the planners based on Model II achieve better downstream planning performance despite their lower test-set prediction accuracy.
Although no single planner consistently dominates in every scenario, these results suggest that the simpler Model II architecture generally leads to better planning outcomes given the current size of the training dataset. However, this advantage may not hold for larger and more diverse training datasets, where the greater representational capacity of Model I may ultimately yield better downstream planning performance. Still, further investigation should therefore assess the generalizability of these findings across a wider range of training scenarios. 

\begin{table}[tbp]
    \footnotesize
    \centering
    \caption{Number of test scenarios in which each planner achieves the lowest weighted objective among all four planners.}
    \label{tab:best_planner}
    \begin{tabular}{lcc}
    \toprule
    Planner & Weighted Cost (TET) & Weighted Cost (TIT) \\
    \midrule
    I-TET  & 20 & 20 \\
    I-TIT  & 10 & 14 \\
    II-TET   & 44 & 50 \\
    II-TIT   & 26 & 16 \\
    \bottomrule
    \end{tabular}
\end{table}

\begin{table}[tb]
    \footnotesize
    \centering
    \caption{Safety-cost prediction errors on each planner's own planned trajectories over all test scenarios.}
    \label{tab:error}
    \begin{tabular}{lccc}
        \toprule
        Model & Label & MAE & RMSE \\
        \midrule 
        Proposed Model I  & TET & 0.6570  & 1.6297  \\
        Proposed Model II & TET & 0.6383  & 1.1963  \\
        Proposed Model I  & TIT & 7.2254  & 17.1961 \\
        Proposed Model II & TIT & 6.5026  & 14.5227 \\
        \bottomrule
    \end{tabular}
\end{table}

Another notable observation is that planners trained using TET-based safety metrics generally achieve lower realized safety costs than those with TIT-based metrics. This result may seem unexpected, given that TIT explicitly accounts for the severity of collision risk. In fact, as defined in Eqs.~\eqref{eq:tet} and~\eqref{eq:tit}, TET measures exposure to potentially unsafe interactions, whereas TIT additionally accounts for risk severity through the reciprocal of TTC, which can produce extreme label values for high-risk trajectories. When models are trained with MSE loss, samples with large prediction errors contribute disproportionately to the training loss. Given limited model capacity and training data coverage, this effect may cause the model to prioritize high-risk samples without achieving sufficient prediction accuracy in lower-risk regions. However, optimized trajectories typically lie in these lower-risk regions, where prediction errors may distort the estimated trade-offs among safety, efficiency, and comfort and thereby degrade the resulting trajectories. Consequently, higher overall prediction accuracy does not necessarily translate into better planning performance.

Overall, the proposed planners achieve substantial improvements in planning feasibility, safety, and weighted objective performance compared with the conventional baseline methods. Among the proposed methods, Model II consistently provides more favorable overall objective values, while the TET-based planner generally leads to lower realized safety costs. These findings highlight the importance of considering not only surrogate prediction accuracy and architectural expressiveness, but also the characteristics of the target safety metric and the quality of the resulting optimization decisions. Consequently, in data-driven optimization frameworks, prediction accuracy alone may not be a sufficient criterion for selecting surrogate architectures and training objectives, and end-to-end planning performance should also be considered.

\section{Conclusions and future work} \label{sec:conclusion}

This study develops a data-driven optimization framework for real-time LC trajectory planning in mixed traffic, in which ICNN-class surrogates are introduced to approximate the safety cost and optimal lateral comfort cost for varying LC trajectories. 
The effectiveness and computational efficiency of the associated LP reformulation are theoretically analyzed and experimentally validated, shedding light on how embedding ICNN-class models can extend the reach of rigorous mathematical optimization to complex, computationally demanding problems. 
Moreover, decoders with different architectures are utilized and tested, demonstrating the flexibility of ICNN-class surrogates. 
The experimental results further indicate that the selection of NN models integrated with optimization methods should be guided by their computational efficiency and performance on downstream tasks instead of predictive accuracy alone. 

Further work should leverage larger and more comprehensive naturalistic datasets, covering a wider range of scenarios. 
In addition, the proposed safety surrogate adopts a relatively simple LSTM module to process historical motion states. More advanced architectures can therefore be considered, since the selection of modules on the non-convex path does not affect the validity or computational efficiency of the proposed surrogate and its LP formulation. Finally, this study adopts TET and TIT to quantify the safety cost of LC maneuvers: although these metrics are widely used in conventional CAV trajectory planning studies, further investigation is required to explore whether more suitable safety metrics exist for data-driven optimization-based LC frameworks.

\section*{Acknowledgments}
We acknowledge the computational resources provided by the Aalto Science-IT project.

\clearpage

{
\small
\bibliographystyle{apalike-ejor}
\bibliography{references_tidy}
}

\end{document}